\documentclass{amsart}
\title[Wind generated water waves]{On the wind generation of water waves} 

\author[O. B\"uhler]{Oliver B\"uhler${}^\dagger$}
\address[O. B\"uhler and J. Shatah]{Courant Institute of Mathematical Sciences \\ New York University \\ 251 Mercer Street, New York, NY 10012}
\email[O. B\"uhler]{obuhler@@cims.nyu.edu}
\thanks{${}^\dagger$Work supported in part by NSF-DMS 1312159}

\author[J. Shatah]{Jalal Shatah${}^\ddagger$}
\email[J. Shatah]{shatah@@cims.nyu.edu}
\thanks{${}^\ddagger$Work supported in part by NSF-DMS 1363013}

\author[S. Walsh]{Samuel Walsh}
\address[S. Walsh]{Department of Mathematics \\ University of Missouri \\  Math. Sciences Building, Columbia, MO 65211} 
\email{walshsa@@missouri.edu}

\author[C. Zeng]{Chongchun Zeng${}^\mathsection$}
\address[C. Zeng]{School of Mathematics\\
Georgia Institute of Technology\\ 686 Cherry Street, Atlanta, GA 30332}
\email{zengch@@math.gatech.edu}
\thanks{${}^{\mathsection}$Work supported in part by NSF-DMS 1362507.}

\usepackage{amssymb,stmaryrd}
\usepackage{amscd}
 \usepackage{mathrsfs,dsfont}
\usepackage[dvips]{graphicx}

\usepackage{geometry}                

\DeclareFontFamily{OT1}{pzc}{}
\DeclareFontShape{OT1}{pzc}{m}{it}{<-> s * [1.10] pzcmi7t}{}
\DeclareMathAlphabet{\mathpzc}{OT1}{pzc}{m}{it}     

\newcommand{\be}{\begin{equation} }
\newcommand{\ee}{\end{equation}}
\newcommand{\bse}{\begin{subequations}}
\newcommand{\ese}{\end{subequations}}

\newcommand{\jump}[1]{\left\llbracket{#1}\right\rrbracket}

\newcommand{\realpart}[1]{\operatorname{Re}{#1}}
\newcommand{\imagpart}[1]{\operatorname{Im}{#1}}

\newcommand{\sgn}[1]{\operatorname{sgn}{#1}}

\newcommand{\matD}{\mathbf{D}_t}

\newcommand{\p}{\partial}
\newcommand{\CN}{\mathcal{N}}
\newcommand{\ep}{\epsilon}
\newcommand{\CW}{\mathcal{W}}
\newcommand{\R}{\mathbb{R}}
\newcommand{\CA}{\mathcal{A}}
\newcommand{\BW}{\mathbf{W}}

\theoremstyle{plain} 
\newtheorem{theorem}{Theorem}[section]
\newtheorem{corollary}{Corollary}[section] 

\newtheorem{lemma}{Lemma}[section]
\newtheorem{proposition}{Proposition}[section] 
\theoremstyle{definition}
 
\theoremstyle{definition} 
 
\theoremstyle{remark} 
\newtheorem{remark}{Remark}[section]

\newtheorem{example}{Example}[section]

\numberwithin{equation}{section}

\usepackage{hyperref}
\hypersetup{ 
colorlinks=true, 
} 

\begin{document}

\begin{abstract}
In this work, we consider the mathematical theory of wind generated water waves.   This entails determining the stability properties of the family of laminar flow solutions to the two-phase interface Euler equation.  We present a rigorous derivation of the linearized evolution equations about an arbitrary steady solution, and, using this, we give a complete proof of the instability criterion of Miles \cite{miles1957windwaves1}.  Our analysis is valid even in the presence of surface tension and a vortex sheet (discontinuity in the tangential velocity across the air--sea interface).  We are thus able to give a unified equation connecting the Kelvin--Helmholtz and quasi-laminar models of wave generation.
\end{abstract}

\maketitle

\section{Introduction}
In this paper, we seek to address the extremely classical problem of determining how wind blowing over the ocean generates waves.  Specifically, our main objective is to give a mathematically rigorous answer to the question: What must the velocity profile of the wind be in order to give rise to persistent waves in quiescent water?  

For our purposes, the air--sea system is modeled by the
two-dimensional incompressible interface Euler problem.   That
is,  we consider the ocean and atmosphere as immiscible fluids,
each evolving according to the free surface Euler equations.  The
assumption of incompressibility is widely adopted and reasonable
because the Mach number for typical flows is quite small (see,
for example, \cite{janssen2004interaction}).  At time $t \geq 0$,
the air inhabits the region $\Omega_t^+$, and the water occupies
$\Omega_t^-$.  The ocean is finite depth with a rigid flat bed.
We also follow the common practice of taking the atmosphere region to
lie below a rigid flat lid; this is justified on the grounds
that, if the flow is evanescent 
at high altitudes, its behavior there does not strongly affect the dynamics near the ocean (see, e.g., 
\cite{janssen2004interaction}).   Letting $S_t := \partial \Omega_t^+ \cap \partial \Omega_t^-$ denote the air--sea interface, we assume moreover that   
\[ \Omega_t^+ \cup \Omega_t^- \cup S_t =: \Omega_t \cup S_t = \mathbb{T} \times (-h_-, h_+).\] 
Here $\mathbb{T} := \mathbb{R} /2\pi \mathbb{Z}$ is the one-dimensional torus.  The constants $h_\pm > 0$ are fixed and describe the location of the ocean bed and atmospheric lid.   For simplicity, we will suppose that $S_t$ is given as the graph of a smooth function $\eta = \eta(t,x_1)$.  This is not strictly necessary, but simplifies many of the computations. 

The velocity field $v = v(t,x) \in \mathbb{R}^2$ and pressure $p = p(t,x) \in \mathbb{R}$ satisfy
 \begin{subequations} \label{intro:eulerianeuler}
 \begin{align} 
  \matD v  + \frac{1}{\rho} \nabla p + g  \mathbf{e}_2 &= 0 \qquad \textrm{in } \Omega_t\label{intro:eulerianmomentum} \\
  \nabla \cdot v &= 0 \qquad \textrm{in } \Omega_t,  \label{intro:euleriandivfree} 
 \end{align}
 where $\matD := \partial_t + \nabla_v$ is the material derivative, $\rho$ is the density, and $g$ is the gravitational constant.  Implicit above is the assumption that $\rho$ is constant in each fluid region.   We point out that the momentum equation \eqref{intro:eulerianmomentum} does not include any turbulent effects --- this is a modeling choice that we discuss below.  
 
 The motion of the interface is dictated by the kinematic condition which expresses the fact that $S_t$ is a material line
\be \partial_t + v_\pm \cdot \nabla \textrm{ is tangent to } \{(t, x) \mid x \in S_t\}. \label{intro:euleriankinematic} \ee
Likewise, the rigidity of the ocean floor and atmospheric lid means that
\be v_\pm \cdot \mathbf{e}_2 = 0 \qquad \textrm{on } x_2 = \pm h_\pm .\ee
Note that we are using the convention that, for a quantity $f$ with domain $\Omega_t$, $f_\pm := f|_{\Omega_t^\pm}$.  
Lastly, we impose the dynamic boundary condition
\be p_+(t,x)-p_-(t,x) = \sigma \kappa(t,x)= \frac {\sigma \eta_{x_1x_1}}{(1+\eta_{x_1}^2)^{\frac 32}} \qquad \textrm{for all } x \in S_t.\label{intro:euleriandynamic} \ee
 \end{subequations}
Here $\sigma \geq 0$ is a (fixed) material constant, and $\kappa(t,x)$ is the signed curvature of $S_t$ at $x$. 

The local well-posedness theory for \eqref{intro:eulerianeuler} has been studied intensively.     Naturally, the irrotational problem has enjoyed the most attention, but as will become clear, vorticity in the atmosphere plays a central role in wind-wave generation.  Thus we will limit our discussion to the literature concerning the rotational case.   With or without vorticity, when surface tension is neglected (i.e., taking $\sigma = 0$), the linearized system is {ill-posed} (cf., \cite{beale1993growth}).  When $\sigma > 0$, the full nonlinear problem is locally well-posed (cf., \cite{cheng2008vortexsheets,shatah2011interface}) 

The interface Euler equations possess a large class of nontrivial exact solutions:  observe that any pair $(v, S_t)$ with the ansatz 
\[ v(t,x) = (U(x_2), \, 0), \qquad S_t = \mathbb{T} \times
\{0\}\] comprises a time-independent solution of
\eqref{intro:eulerianeuler}.  These are called \emph{laminar} or
\emph{shear flows}, and serve as a model for the state of the
air--sea system before water waves have formed.  If the
background flow is stable, then perturbations will remain in a
neighborhood of the equilibrium, meaning that the interface does
not leave its quiescent state.  On the other hand, if it is
unstable, then the free surface will become 
deformed at 
finite amplitude, i.e., persistent surface waves will be born.
In this way, the study of wind generation of water waves is
equated with diagnosing the stability/instability of the laminar
flows subject to Eulerian dynamics.

A natural starting point is to consider the situation where the ocean is at rest and the velocity of the air is uniform, i.e., $U_- \equiv 0$, $U_+ \equiv U_0$, for some constant $U_0$.  This is the classical Kelvin--Helmholtz model, and it is indeed (linearly) unstable whenever $\sigma = 0$ and $U_0 \neq 0$ (in a sense that we will make precise later).   With surface tension, the flow is (linearly) stable for $U_0$  satisfying the inequality
\[ U_0^2 \leq 2\frac{\rho_+ + \rho_-}{\rho_+ \rho_-}  [g \sigma (\rho_- - \rho_+)]^{1/2}.  \]
(See, e.g., \cite{drazin2004book}).  Setting $\rho_\pm$ and
$\sigma$ to their physical values, this predicts that the onset
of instability occurs when $U_0 > 6.6$ m/s, which is an order of
magnitude larger than observation suggest.  Worse still, this
instability is first felt at very small wavelength, roughly
$0.017$m; to excite a wave with a more typical wavelength 
of less than 100m, say, requires $U_0$ to be another order of
magnitude larger.

One is forced to admit, therefore, that the Kelvin--Helmholtz
instability fails miserably as a model of the wind generation of
ocean waves.  In particular, it is missing some destabilizing
mechanism inherent in the physical system.  Kelvin himself
observed this fact in his original article on the topic in 1871
(cf. \cite{kelvin1871hydro}).  Since then, the task of finding a
suitable replacement has been a fundamental problem in
geophysical fluid dynamics. The next century saw a succession of
competing models (cf., e.g., \cite{janssen2004interaction} for a
summary), but now the majority opinion has largely settled on the
\emph{quasi-laminar model}.  Put forward by Miles in a series of
papers
(cf. \cite{miles1957windwaves1,miles1959windwaves2,miles1959windwaves3}),
the quasi-laminar model views the wind generation process as a
resonance phenomenon.  Briefly, the idea is the following.  Since
$\rho_+/\rho_- \ll 1$, the atmosphere may be viewed as a
perturbation of vacuum.  For infinite depth gravity water water
beneath vacuum, one has the dispersion relation $c = \sqrt{g/k}$,
where $c$ is the wave speed and $k$ is the wavenumber.  That is,
the formally linearized problem has an exact solution for which
$S_t$ is given as the graph of a function proportional to
$e^{ik(x_1-ct)}$. This suggests that the dispersion relation with
an atmosphere may have the asymptotic form \be c =
\sqrt{\frac{g}{k}} + c_1 \epsilon + O(\epsilon^2), \qquad
\epsilon := \frac{\rho_+}{\rho_-}.\label{intro:milesdispersion}
\ee In view of \eqref{intro:milesdispersion}, the situation where
a \textsl{critical layer} exists is of special interest: in this
case the unperturbed phase speed $\sqrt{g/k}$ lies in the range
of flow speeds $U_+$ and this may enable a certain resonance
between the shear flow in the atmosphere and the gravity waves in
the ocean, which then manifests itself as a linear instability.
Formally linearizing \eqref{intro:eulerianeuler} and searching
for a growing mode solution of the form $v = \psi(x_2)
e^{ik(x_1-ct)}$, Miles concluded that (linear) instability occurs
if 
a critical layer location $x_2^*$ exists where $\sqrt{g/k} =
U(x_2^*)$ and $U''(x_2^*)<0.$  This computation was done with
$\sigma = 0$, meaning that the Cauchy problem for the full system is ill-posed.  Nonetheless, if one fixes a physically realistic value
for the wavenumber $k$, one can still study the linear evolution.  Doing this, Miles's scheme gives a way to estimate
the corresponding growth rate.  Most importantly, in contrast to
Kelvin--Helmholtz, the quasi-laminar model allows for wave
generation without unreasonably large wind speed.

Though it is now widely accepted, the quasi-laminar model has been criticized in the applied literature for consigning turbulence to a relatively minor role (cf., e.g.,  the discussion in \cite[Chapter 4]{janssen2004interaction}).   Indeed, turbulence is absent in the basic equations \eqref{intro:eulerianeuler}, and so its influence is felt only implicitly.  For instance, Miles uses a logarithmic wind profile in his growth rate computations, which is predicted by the theory of turbulent laminar flow over a flat plate.  Also, the presence of a turbulent boundary layers is used to justify his assumption that $U_+(0) = U_-(0) = 0$.  Since the actual air flow near the ocean surface is known to be highly turbulent, many authors have argued that the effects of turbulence must be included explicitly in the dynamics of the air flow.      Nonetheless, both in field observations and laboratory studies \cite{plant1982relationship,hristov2003dynamical}, the quasi-laminar model has proven
  to be a rather good predictor of wind energy transfer in many settings.  Of course, when considering these studies, one must take into account the extreme difficulty in obtaining accurate readings; even the task of deducing the wind profile is highly nontrivial.  In short, there is a great deal of uncertainty, but the ideas of Miles have been more-or-less borne out by the subsequent half-century of experimentation.    

 \subsection*{Summary of results}
 
 Having established the background, let us now enumerate the mathematical contents of the present work.  It consists of two parts.  The first is a careful derivation of the linearized interface Euler system \eqref{intro:eulerianeuler} about an arbitrary steady solution, including those with a vortex sheet and taking into account surface tension.  The second is a rigorous analysis of the behavior of solutions of the Rayleigh stability equation when the background flow has a critical layer occurring at an non-inflection point.  We elaborate on both of these points below.   
 
 \subsubsection*{Rigorous linearization}
All prior derivations of the linearized equations for the two phase Euler flow are done heuristically and in a setting that does not allow surface tension and/or a vortex sheet if the background velocity is non-uniform.  In \cite{janssen2004interaction}, for instance, Janssen suggests that the two-fluid problem  be imagined as a limit of single-fluid problems with smooth heterogeneous density.    The advantage of this approach is that it obviates the need to fix the domain, since the one phase fluid simply occupies the entire plane. However, the limiting process as the density becomes discontinuous over the interface is not straightforward.  It is hard, for instance, to see even at a formal level how the equations governing the dynamics of the free boundary arise when one allows for a vortex sheet and surface tension.  Moreover, rigorously establishing that the limit of the one phase problems converge to a solution of the two phase problem is a very difficult task (cf., e.g., \cite{james2001internal} where it is carried out in a very specific physical regime).

Another typical approach is to linearize the Euler equation \eqref{intro:eulerianmomentum} separately in each fluid region, then formally linearize the equation of boundary motion \eqref{intro:euleriankinematic}.  This is precisely what Miles does in \cite{miles1957windwaves1}, where again we note that he assumes the  continuity of the tangential velocity across the interface and works without surface tension.  A flow is then considered linearly unstable if there exists a growing mode solution of the resulting problem.  

There is one feature that this process lacks to be mathematically rigorous.    Notice that we are considering a free boundary problem, so the meaning of ``linearized operator'' is somewhat subtle.  From a dynamical systems point of view, the only suitable definition of the linearized propagator is found by taking the Fr\'echet derivative or G\^ateaux variation.  For this to make sense, the linearized problem must first be formulated in a specific function space.  In the present setting, this issue manifests itself when one takes variations of the velocity field.  Performing the  standard formal linearization of the Euler equation, we can see immediately that the linearized velocity will have the wrong boundary conditions on the interface and hence is not in the correct function space.  This is expanded upon in Section \ref{S:LE};  specifically \eqref{general LV1} and Remark \ref{linearization remark}.  

Our linearization based on the Hodge decomposition leads us to a system consisting of an ODE, Rayleigh's equation, coupled to an algebraic equation related to the dispersion relation.  On the one hand, in the absence of surface tension and a vortex sheet, this coincides with the one derived formally by Miles.  Thus our results mathematically corroborate and generalize his. On the other hand, taking the background flow to be uniform in the air and water regions, we recover the classical Kelvin--Helmholtz instability criteria (see Example \ref{kelvin helmholtz ex}).  This means that we can consider simultaneously the destabilizing effects of a vortex sheet and a critical layer, and  we can see quantitatively the individual influence of each one on the leading-order terms of the unstable eigenvalue.

\subsubsection*{A proof of Miles's criterion and generalizations}
Using our linearized problem, we look for an unstable eigenvalue.  At this point, in \cite{miles1957windwaves1}, Miles proceeds with the heuristic argument outlined above, concluding that a linear instability will manifest at wavenumber $k$ provided that $\sqrt{g/k}$ is in the range of the wind profile $U_+$, and that it occurs at an altitude where $U_+^{\prime\prime}$ is negative.   

Our presentation is the first rigorous treatment of the system studied by Miles and its generalization that we have derived.   Mathematically, the analysis involves a detailed examination of the Rayleigh  equation in the presence of an imaginary parameter --- the complex wave speed $c$ --- that couples the ODE to the dispersion relation.  

The Rayleigh equation (also called the inviscid Orr--Sommerfeld equation) arises when investigating the stability of laminar flows in a rigid channel and has an extensive literature.  However, the challenges we confront here are quite distinct from those typically encountered.  For shear flows in a channel, the boundary conditions are simply homogeneous Dirichlet, whereas the free surface in our system leads to an inhomogeneous condition on the interface (see also \cite{hur2008unstable}).   More significantly, for channel flows, most rigorous sufficient conditions for instability are based on bifurcation analysis with $c$ near an inflection point of $U$. In such cases, treating the wave number as the bifurcation parameter, one finds that the leading-order part of the Rayleigh equation is in fact not singular. The situation is almost the complete opposite in our work:   the critical layer instability occurs precisely when $c$ is close to some $U(x_2^*)$ with  $U''(x_2^*)\ne 0$. This means that the crucial part of our analysis is near the singularity of the Rayleigh equation where the solutions will develop a logarithmic singularity in the limit $\realpart{c} \to 0$. One may find plenty of asymptotic expansions of the Rayleigh equation solutions near such a singularity in the literature, some of which can be rigorously justified provided that $U$ is analytic \cite{wasow1987asym-ode}. By contrast, our approach based on more modern tools from dynamical systems, requires only that $U_+ \in C^4$ and is transparently rigorous.    Along the way, we obtain a result of independent interest characterizing the limiting behavior of solutions to the Rayleigh equation as the wave speed $c$ converges to the real axis; see Proposition \ref{P:Rayleigh}.  

This analysis is then used to prove the main contribution found in Theorem \ref{T:Instability}.  For a fixed wave number $k$, let $c_k$ denote the wave speed for the corresponding (linear) capillary-gravity wave in finite-depth water beneath vacuum (cf. \eqref{E:ck}).  We obtain a sufficient condition guaranteeing the existence of an unstable wave speed lying in a neighborhood of $c_k$ that is valid even with a vortex sheet and incorporating surface tension. We are, moreover, able to determine the leading-order terms of the unstable eigenvalue, which gives a means of estimating the energy transfer rate; see \eqref{def cpound} and \eqref{c expansion}.   Additionally, in Proposition \ref{P:evalue}, we obtain an exact (though implicit) formula for the dispersion relation analogous to \eqref{intro:milesdispersion}.   All of these results agree with the formal analysis of Miles when the background velocity is continuous over the interface and surface tension is neglected. Moreover, in Lemma \ref{L:necessary}, we show that if a sequence of the unstable wave speed $c_{k, \ep}$ converges to $c_k$ as $\ep = \frac {\rho_+}{\rho_-} \to 0$, then $c_k$ must be in the in the range of $U_+$.  Roughly speaking, this says that a critical layer is necessary for the generation of instability. 
 
It is also worth comparing the linear instability in the channel flow problem with the water-air interface problem. In the case of channel flow between rigid boundaries, where the wave number is often used as the parameter, the instability occurs at isolated wave numbers and thus the linearly unstable waves are superpositions of plane waves of isolated wave numbers. On the contrary, for the water-air interface problem considered here, the instability due to critical layers occurs at all wave number $k$ in certain intervals (or a union of intervals). In particular, with surface tension, sufficiently large or small wave numbers are always stable. Without surface tension, sufficiently long waves are always stable, while the instability of short wave is affected by both the critical layers and Kelvin--Helmholtz instability. The linearly unstable waves are superpositions of waves with wave numbers ranging in an interval. 

\subsubsection*{Unstable waves with critical layer at an inflection point}

The critical layer analysis is valid under the assumption that the shear profile $U$ is of class $C^4$ in the air region.  As an application of the exact dispersion relation, in Section \ref{pathological example section} we construct an explicit piecewise linear $U$ (thus $U^{\prime\prime}$ is a $\delta$-function)  that are linearly unstable.  However, it is shown that, for a certain range of parameters, $c_k$ is in the range of $U$, but occurs away from the mass of the $\delta$-function, hence $U^{\prime\prime}$ vanishes there.    This implies that the instability is \emph{not} arising from resonance with the critical layer but instead from the immensity of $|U^{\prime\prime}|$.   For such waves, we have that wave speed $c$ is at an $O(1)$ distance from $c_k$, whereas the expansion \eqref{intro:milesdispersion} is only valid up to $O(\epsilon)$.  In fact, demonstrates the sharpness of our necessary condition for instability in Lemma \ref{L:necessary}: there we require that $U_+ \in C^2$, whereas these profiles are of class $C^{0,1}$.  

\section{Linearization} \label{S:LE}

In this section, we will derive the linearization of the Euler equation \eqref{intro:eulerianeuler} about a steady solution with graph geometry.  That is, we consider a solutions of the form 
\be \label{E:steadysolution}
S_t^* = S = \{ x_2 = \eta^*(x_1) \}, \qquad v = v^*(x_1, x_2),
\ee
for some smooth wave profile $\eta^*$.  One should think of $(v^*, S)$ as representing a traveling wave observed from a moving reference frame so that it appears stationary.    For the quasi-laminar model, we are specifically interested in shear flows, i.e. solutions of the form 
\begin{equation}\label{E:steadyshear} 
 S= \{ x_2 =0\}, \qquad  v^* ( x_1, x_2) = \big(U(x_2), 0\big), \quad x_1 \in \mathbb{T}, \;  x_2 \in [-h_-, h_+]
\end{equation}
where 
\[
U(x_2) = U_-(x_2) \chi_{\{x_2<0\}} + U_+(x_2) \mathbb{\chi}_{\{x_2 >0\}}
\]
and $U_\pm$ are smooth functions on $\pm x_2 \in [0, h_\pm]$. Note that we are allowing  $v^*$ to have a jump discontinuity over $S$.  Eventually we will consider those shear flows with $U_- \equiv 0$, i.e. the water will be assumed to be stationary while there is wind in the air.  For those shear flows, the pressure is hydrostatic, 
\begin{equation} \label{E:steadypressure}
\nabla p_{\pm}^* = - g \rho_\pm \mathbf{e}_2.
\end{equation}
It is elementary to confirm that all shear flows are solutions of the Euler system.  The existence of traveling waves where $S$ is not flat has been established by many authors in various regimes (cf., e.g., \cite{amick1986global,walsh2013steadywind}).

\subsection*{Admissible spaces and orthogonal decompositions} Before we begin, we must introduce the spaces in which the problem is formulated.  There are several results that prove the local well-posedness of the Cauchy problem for the interface Euler equation \eqref{intro:eulerianeuler} (cf. \cite{cheng2008vortexsheets,shatah2011interface}).  These consider smooth velocity fields and surface profiles, for example $v \in H^s(\Omega_t^\pm)$ for $s >  5/2$.  To establish linear instability, we will assume that the background flow is smooth and seek solutions of the linearized problem that grow in the $L^2$ norm. 

With that in mind, for any time $t \geq 0$, we consider velocity fields $v$ belonging to the $S_t$-dependent space 
\[\begin{split} 
\mathbb{X}(S_t) := \{ v: L^2(\mathbb{T} &\times [-h_-, h_+], \mathbb{R}^2; \rho \, dx) \mid   \nabla\cdot v_\pm =0 \text{ in } \Omega_t^\pm, \; \,\\
& v_\pm \cdot \mathbf{e_2} =0 \text { on } x_2 =\pm h_\pm, \; \, v_+ \cdot N_+ + v_- \cdot N_-=0 \text{ on } S_t\} 
\end{split}\]
where $N_\pm$ are the unit outward normals for $\Omega_t^\pm$.  Note that they obviously satisfy $N_+ + N_-=0$. 
The boundary conditions on $S_t$ included in the definition of $\mathbb{X}(S_t)$ are meant to guarantee that $\nabla\cdot v=0$ in the distribution sense on $\mathbb{T} \times [-h_- , h_+]$. For a divergence free vector field $v$ in $L^2$, its normal component $v \cdot N$ is well-defined in $H^{-\frac 12}(S_t)$, as $S_t$ is smooth (cf., e.g., \cite{temam1977book}).   $\mathbb{X}(S_t)$ is a subspace of $L^2 (\mathbb{T} \times [-h_- , h_+],  \rho \, dx)$ and  its orthogonal complement  $\mathbb{X}(S_t)^\perp$ is given by 
\[
\mathbb{X}(S_t)^\perp = \{ v = - \nabla q \mid q= q_+ \chi_{\Omega_t^+} + q_- \chi_{\Omega_t^-}, \; \, q_\pm \in H^1(\Omega_t^\pm) \text{ and } \rho_+ q = \rho_- q \text{ on } S_t\}.
\]
This can be seen from the Hodge decomposition as described in \cite{shatah2008interfaceapriori, shatah2011interface}. For any $X \in L^2 (\mathbb{T} \times [-h_- , h_+],  \rho \, dx)$,  
\begin{equation} \label{E:decom1}
\textrm{there exists } \; w \in \mathbb{X}(S_t) \text{ such that } X =w -\nabla q \text{ and } -\nabla q\in \mathbb{X}(S_t)^\perp
\end{equation}
where $q$ is determined (uniquely up to a constant) by 
\begin{equation} \label{E:decom2} \begin{cases}
-\Delta q = \nabla\cdot X \qquad \big(\mathbb{T} \times (-h_- , h_+)\big) \backslash S_t\\
q_\pm|_{S_t} = \frac {1}{\rho_\pm}q^S := - \dfrac {1}{\rho_\pm} \CN^{-1}  \big( N_+ \cdot ( X_+ - \nabla \Delta_+^{-1} \nabla\cdot X_+) +N_- \cdot ( X_- - \nabla \Delta_-^{-1} \nabla\cdot X_-) \big)\\
\nabla q_\pm \cdot \mathbf{e}_2 |_{x_2 = \pm h_\pm} = X_\pm \cdot \mathbf{e}_2\end{cases}\end{equation}
Here we write $(\Delta_\pm)^{-1}$ to denote the inverse Laplacian on $\Omega_t^\pm$ with zero Dirichlet boundary condition on $S_t$ and zero Neumann boundary condition on $\{x_2 = \pm h_\pm\}$.  Likewise, we let 
\[
\CN :=\frac 1{\rho_+}\CN_+ + \frac 1{\rho_-}\CN_-
\]
with $\CN_\pm$ being the Dirichlet-to-Neumann operator on $S_t$ associated to $\Delta_\pm$ with zero Neumann data on $\{x_2 = \pm h_\pm\}$. Again, we note that for $X$ as above, 
\[ N_\pm \cdot ( X_\pm - \nabla \Delta_\pm^{-1} \nabla\cdot X_\pm) \in H^{-\frac 12} (S_t),\]
since $X_\pm - \nabla \Delta_\pm^{-1} \nabla\cdot X_\pm$ is in $L^2 (\Omega_t^\pm)$ and divergence free while $S_t$ is assumed to be smooth. 

\subsection*{Linearized evolution equation} To compute the linearized equations, let $\big(S_t(\alpha), v(\alpha, t, x)\big)$ along with the pressure $p(\alpha, t, x)$, be a one-parameter family of solutions of the Euler  equation \eqref{intro:eulerianeuler} that coincide with the steady state $(S, v^*)$ at $\alpha = 0$. Since $S_t$ is close to $S$ (the graph of $\eta^*$) for small $\alpha$, we may represent $S_t$ as the graph of a unique function $\eta$ with $\eta|_{\alpha = 0} = \eta^*$:
\[
S_t =\{x_2 = \eta(\alpha, t, x_1), \; x_1 \in \mathbb{R}\}.
\]
Eventually we will specialize this to the case of a shear flow that is periodic in $x_1$, but for now we work in the general setting.  

From the boundary condition \eqref{intro:euleriandynamic}, we have (suppressing the dependence on $t$ and $\alpha$)
\[
p_+ \big(x_1, \eta(x_1)\big) - p_-\big(x_1, \eta(x_1)\big) = \sigma \kappa = \frac{\sigma \eta_{x_1x_1}}{(1+\eta_{x_1}^2)^{3/2}}.
\]
Let $z (t, x_1) := (\partial_\alpha \eta)|_{\alpha=0}$, which is the component of the linearized solution corresponding to the variation of the interface. Differentiating the above equality we obtain
\begin{equation} \label{E:zeq1}
\sigma \kappa^\prime(\eta^*) z = \p_\alpha p_+ \big(t, x_1, \eta^*(x_1)\big) - \p_\alpha p_- (t, x_1, \eta^*(x_1)\big) + (\p_{x_2} p_+^* - \p_{x_2} p_-^*) (x_1, \eta^* (x_1) ) z
\end{equation}
where 
\[ \kappa^\prime(\eta^*)  := \frac{1}{(1+(\eta_{x_1}^*)^2)^{3/2}} \partial_{x_1}^2 - 3 \frac{\eta_{x_1 x_1}^*}{(1+(\eta_{x_1}^*)^2)^{5/2}} \eta_{x_1}^* \partial_{x_1}.\] 

Since $v(\alpha, t) \in \mathbb{X}(S_t(\alpha))$, it follows that
\[
\Big( v_+ \big(x_1, \eta(x_1) \big) - v_- \big( x_1, \eta(x_1)\big) \Big) \cdot (- \p_{x_1} \eta , 1)^T =0.
\]
In the above equation, we are again suppressing the dependence on $\alpha$ and $t$.  Differentiating in $\alpha$ and evaluating at $\alpha =0$, we obtain 
\be \left( \jump{ \p_\alpha v} + \jump{ \partial_{x_2} v^*} z \right) \cdot (-\partial_{x_1} \eta^*, 1)^T  =  \jump{v^* \cdot \mathbf{e}_1}  \p_{x_1} z \quad \text{ on } S.   
 \label{general LV1} \ee
 Here, for a function $f$ defined on $\mathbb{T} \times [-h_- , h_+]$, we write $\jump{f} := (f_+ - f_-)|_S$. Observe that this computation shows that in general the linearized velocity field $\p_\alpha v$ is \emph{not} in $\mathbb{X}(S)$ even though $\p_\alpha v_\pm$ is divergence free in $\Omega_\pm$ and $\p_\alpha v_\pm \cdot \mathbf{e}_2=0$ along $\{x_2=\pm h_\pm\}$; see also Remark \ref{linearization remark}.   Our next step is therefore to decompose $\p_\alpha v$ into two components, one lying in $\mathbb{X}(S)$ and the other in $\mathbb{X}(S)^\perp$:
\[
\p_\alpha v = Y + \nabla \Gamma, \quad Y \in \mathbb{X}(S) \text{ and } \nabla \Gamma \in \mathbb{X}(S)^\perp.
\]
From \eqref{E:decom1}, \eqref{E:decom2}, and \eqref{general LV1} 
\be\label{general gammaeq}  \left\{ \begin{array}{ll} \Delta \Gamma = 0 \quad \qquad \textrm{in }  \Omega \\ 
\partial_{x_2} \Gamma_\pm = 0 \qquad \textrm{on } \{x_2 = \pm h_\pm\}, \\ 
\Gamma_\pm = \dfrac{1}{\rho_\pm} \mathcal{G}^{-1} \left( \jump{ \partial_{x_2} v^*} z \cdot (-\partial_{x_1} \eta^*, 1)^T -  \jump{v_1^*} \partial_{x_1}  z \right)  \qquad \textrm{on } S,  \end{array} \right.  \ee
where $\mathcal{G}$ is the non-normalized Dirichlet--Neumann operator for $S$ in $\mathbb{T} \times [-h_- , h_+]$:
\[ \mathcal{G}_\pm := \sqrt{1+(\eta_{x_1}^*)^2} \mathcal{N}_\pm, \qquad \mathcal{G} := \frac{1}{\rho_+} \mathcal{G}_+ + \frac{1}{\rho_-} \mathcal{G}_-.\]

Likewise, linearizing the momentum equation \eqref{intro:eulerianeuler} we find that $\p_\alpha v= Y + \nabla \Gamma$ satisfies
\be \label{general Yeq} Y_t + (v^* \cdot \nabla) Y + (v^* \cdot \nabla) \nabla \Gamma +  (Y \cdot \nabla) v^* + (\nabla \Gamma \cdot \nabla) v^* + \nabla P = 0 \qquad \textrm{in } \Omega,  \ee
where $Y=(Y_1, Y_2)^T$ and 
\[
P = \frac 1\rho p_\alpha + \Gamma_t. 
\]
Since $Y \in \mathbb{X}(S)$, taking the divergence and the normal component along $x_2 = \pm h_\pm$ of \eqref{general Yeq} and using \eqref{E:zeq1} and \eqref{general gammaeq}, we can determine $P$ by solving
\be \label{general Peq} \begin{cases} 
-\Delta P =  \nabla \cdot \left(  (v^* \cdot \nabla) Y + (v^* \cdot \nabla) \nabla \Gamma +  (Y \cdot \nabla) v^* + (\nabla \Gamma \cdot \nabla) v^* \right) \qquad \textrm{in } \Omega \\
\partial_{x_2} P = 0 \qquad \textrm{on } \{ x_2 = \pm h_\pm \} \\
\rho_+ P_+ - \rho_- P_- =  
- (\p_{x_2} p_+^* - \p_{x_2} p_-^*) z+\sigma \kappa^\prime(\eta^*) z  \quad \text{ on } S
\end{cases}\ee

Finally, because $S_t$ is given as the graph of $\eta$, the fact that the normal velocity of the fluid interface coincides with the normal component of the velocity field along the interface translates to the following statement
\[
\eta_t (x_1)  = v(x_1, \eta(x_1)) \cdot (-\eta' (x_1), 1)^T.  
\]
Linearizing this equality gives
\be z_t + (v_\pm^* \cdot \mathbf{e}_1)  \partial_{x_1} z = \left( \partial_\alpha v_\pm +  \partial_{x_2} v_\pm^* z \right) \cdot (-\partial_{x_1} \eta^*, 1)^T \qquad \textrm{on } S.
\label{general zeq2} \ee 
Due to \eqref{general LV1}, the above equation does not depend on the choice of $+$ or $-$ sign. 

Evolution equations \eqref{general Yeq} and \eqref{general zeq2} along with \eqref{general gammaeq} and \eqref{general Peq} form the linearized system of the two phase fluid Euler equation at an arbitrary steady solution $(v^*, S)$ with graph geometry.   

In the present work, we are mainly interested in the stability of shear flows under periodic perturbations.  It is therefore useful to  record how these equations simplify for such flows.  If we take $(v^*, S)$ to be given as in \eqref{E:steadyshear}, then in particular $S = \mathbb{T} \times \{0\}$, so $\mathcal{G} = \mathcal{N}$ and $\kappa^\prime(0) = \partial_{x_1}^2$.  Thus \eqref{general Yeq} and \eqref{general zeq2} become
\begin{align}
Y_t + U Y_{x_1} + U' Y_2 \mathbf{e_1} + U \nabla \Gamma_{x_1} + \Gamma_{x_2} U' \mathbf{e}_1  +\nabla P  = 0 &  \qquad \textrm{in } S \label{E:Yeq} \\
z_t + U_\pm (0) z_{x_1} = \p_\alpha v_\pm \cdot \mathbf{e}_2 =  Y_2  + \p_{x_2}\Gamma_\pm  & \qquad  \textrm{on } \{x_2 = 0\}, \label{E:zeq2} 
 \end{align}
where $\Gamma$ and $P$ are determined from 
\be \left\{ \begin{array}{ll} \Delta \Gamma = 0 \quad \qquad \textrm{in } \Omega \\ 
\partial_{x_2} \Gamma_\pm |_{x_2 = \pm h_\pm} = 0 \qquad 
\Gamma_\pm|_S = \frac 1{\rho_\pm} \big(U_-(0) - U_+(0)\big) \partial_{x_1} \mathcal{N}^{-1} z,  \end{array} \right. \label{gammaeq} \ee
and
\be \label{E:Peq} \begin{cases} 
-\Delta P = 2 U' (\partial_{x_1} Y_2 + \Gamma_{x_1x_2}) \qquad \textrm{in } \Omega \\
\partial_{x_2} P  =0 \qquad \textrm{on } \{ x_2 = \pm h_\pm\} \\
 \rho_+ P_+ - \rho_- P_- =  g (\rho_+ - \rho_-) z +\sigma z_{x_1x_1} \qquad \text{ on } S =\{x_2=0\}.
\end{cases}\ee

\subsection*{Eigenvalues and eigenfunctions} Notice that the coefficients in the linearized system \eqref{E:Yeq}, \eqref{E:zeq2}, \eqref{gammaeq}, and \eqref{E:Peq} depend only on $x_2$.  Therefore, each Fourier mode $e^{ikx_1}$ is decoupled from the other modes.  Consider solutions of the linearized system taking the form 
\be \label{E:EM1}
(z, Y, \Gamma, P) = \big( \mathpzc{z},  \mathcal{Y} (x_2),  \gamma (x_2), \mathcal{P}(x_2) \big) e^{ik(x_1 - ct)}, \qquad k \in \mathbb{Z}\backslash \{0\},
\ee
which represents and eigenfunction for the linearized system corresponding to an eigenvalue $-ikc$. Clearly, the existence of a solution of this type with  $\imagpart{c} >0$ immediately implies exponential linear instability. 

In what follows, we derive a linear system for the above unknowns, fixing the Fourier mode $k$.  Before doing that, let us record the symbol for the Dirichlet--Neumann operators that we employ:
\be \label{E:CN} \begin{split}
\widehat{\mathcal{N}_\pm}(k) & = |k| \tanh{(|k|h_\pm)} \\
\widehat{\mathcal{N}}(k) & = \frac{1}{\rho_+} \widehat{\mathcal{N}_+}(k) + \frac{1}{\rho_-} \widehat{\mathcal{N}_-}(k) = |k| \left(\frac 1{\rho_+} \tanh{(|k|h_+)} + \frac 1{\rho_-} \tanh{(|k|h_-)}\right).
\end{split} \ee

We can then solve \eqref{gammaeq} to find
\be  \begin{split} 
{\gamma}_\pm(x_2) &= \frac{ik (U_-(0) -U_+(0)) \mathpzc{z}}{\rho_\pm \widehat{\mathcal{N}}(k) \cosh{(h_\pm |k|)}} \cosh{(|k|(x_2 \mp h_\pm))}\\
&= \frac{\rho_\mp (U_-(0) -U_+(0)) }{|k| \big(\rho_- \tanh{(|k|h_+)} + \rho_+ \tanh{(|k|h_-)}\big)}  \frac {\cosh{(|k|(x_2 \mp h_\pm))}} {\cosh{(h_\pm |k|)}} ik \mathpzc{z}.
\label{gammahatformula} \end{split} \ee
Substituting \eqref{E:EM1} and \eqref{gammahatformula} into \eqref{E:zeq2} (with the $+$ sign at $x_2=0$), we obtain 
\[\begin{split} 
ik (U_+(0) - c)  \mathpzc{z} &= \mathcal{Y}_2(0) - \widehat{\mathcal{N}_+}(k)  \gamma_+ (0)  = \mathcal{Y}_2(0) -  |k| \tanh{(|k|h_+)}  \gamma_+ (0)\\
&= \mathcal{Y}_2(0) - \frac{\rho_- (U_-(0) -U_+(0)) \tanh{(|k| h_+)}} {\rho_- \tanh{(|k|h_+)} + \rho_+ \tanh{(|k|h_-)}}ik \mathpzc{z}
\end{split}\]
which implies 
\be \label{E:zhatformula} 
ik \Big( \frac {\rho_+ U_+(0) \tanh{(|k|h_-)} + \rho_- U_-(0) \tanh{(|k|h_+) }}{\rho_- \tanh{(|k|h_+)} + \rho_+ \tanh{(|k|h_-)}} - c \Big) \mathpzc{z} = \mathcal{Y}_2(0).
\ee
Again we recall that because $Y \in \mathbb{X}(S)$, $\mathcal{Y}_{2+}(0)=\mathcal{Y}_{2-}(0)$. 

In the next step, we will use the fact that $Y$ is divergence free, along with its boundary behavior, to eliminate $\mathcal{P}$ and $\mathcal{Y}_1$, obtaining an equation for $(\mathcal{Y}_2, \mathpzc{z})$.  Notice, for instance, that because $Y$ is divergence free,
\[ ik \mathcal{Y}_1 + \p_{x_2} \mathcal{Y}_2 =0.\]
Thus $\mathcal{Y}_1$ can be determined from $\mathcal{Y}_2$.  In light of this observation, the horizontal component of \eqref{E:Yeq} becomes 
\be \label{E:Y1hat}
- (U- c)  \mathcal{Y}_2^\prime  +  U'  \mathcal{Y}_2 - k^2 U  \gamma +  \gamma^\prime U^\prime  +ik \mathcal{P}=0.
\ee
On the other hand, the vertical component of \eqref{E:Yeq} implies 
\be \label{E:YVertical}
ik (U -c) \mathcal{Y}_2 + ik U \gamma^\prime+ \mathcal{P}^\prime=0. 
\ee
Using the above two equations, we can eliminate $\mathcal P$, obtaining the ODE 
\be \label{E:Y2hatformula}
- \mathcal{Y}_2^{\prime\prime}  + (\frac {U''}{U-c} + k^2) \mathcal{Y}_2 + \frac {U''}{U-c} \gamma^\prime =0,
\ee
where the prime denotes $\partial_{x_2}$ and we have used the fact $\Delta \Gamma=0$.  As $Y \in \mathbb{X}(S)$, we have moreover that  $\mathcal{Y}_{2\pm} (\pm h)=0$. The behavior of  $\mathcal{Y}_2$ on the interface is dictated by \eqref{E:zhatformula}. Recall also that the boundary behavior of $P$ is given by \eqref{E:Peq}, which, together with \eqref{E:Y1hat}, implies 
\be \label{E:ceq}\begin{split} 
\left(g \jump{\rho} - \sigma k^2\right) \mathpzc{z} &=  \jump{\rho \mathcal{P}} \\
&= \frac ik \left( - \jump{\rho (U -c)  \mathcal{Y}_{2}^\prime}  + \jump{\rho U'}  \mathcal{Y}_{2}  - k^2 \jump{U}\rho_+\gamma_+  + \jump{ \rho U'  \gamma^\prime} \right),
\end{split}\ee 
as $\rho_+ \gamma_+ = \rho_-  \gamma_-$ on $S$.  

In summary, we have the following result.

\begin{proposition} \label{P:linear}
$-ikc$ is an eigenvalue of the linearized two phase fluid Euler equation if there exist nontrivial linearized solutions of a single Fourier mode of the form  \eqref{E:EM1}.  This is equivalent to the existence of a nontrivial solution $(\mathpzc{z}, \mathcal{Y}_2,  \gamma, c)$ to \eqref{gammahatformula}, \eqref{E:zhatformula}, \eqref{E:Y2hatformula}, \eqref{E:ceq} and such that $\mathcal{Y}_{2\pm}(\pm h) =0$. 
\end{proposition}

\begin{remark} 
The above calculation is still valid if one or both of $h_\pm$ becomes infinity. For example, if $h_-=\infty$, then defining $\tanh (|k|h_-)=1$, we get 
\begin{gather*}
 \gamma_-(x_2) = \frac{\rho_+ (U_-(0) -U_+(0)) e^{|k|x_2} }{|k| \big(\rho_- \tanh{(|k|h_+)} + \rho_+ \big)} ik \mathpzc{z}, \\ \mathcal{Y}_2(0)=ik \Big( \frac {\rho_+ U_+(0) + \rho_- U_-(0) \tanh{(|k|h_+) }}{\rho_- \tanh{(|k|h_+)} + \rho_+} - c \Big) \mathpzc{z} . 
\end{gather*}
\end{remark} 

\begin{remark}  Note that, up to this point, we have not exploited any small parameters in the problem.  So, for instance, Proposition \ref{P:linear} holds even for a system consisting of two fluids with roughly equal density, as one would expect with internal waves moving through a channel.  Of course, the heart of the analysis to come is in determining $\mathcal{Y}_{2+}^\prime(0)$, and for this we will make strong use of the assumption that $\rho_+/\rho_- \ll 1$.
\end{remark}

\begin{remark} \label{R:deltamass}
While $U''$ appears in equation \eqref{E:Y2hatformula}, it is actually not necessary to assume $U_\pm^{\prime\prime}$ exists in the strong sense. In fact,  in the hypotheses of Proposition \ref{P:linear}, one may replace \eqref{E:Y2hatformula} by the following equation derived from  \eqref{E:Y1hat} and \eqref{E:YVertical}:
\[
\Big(- (U- c)  \mathcal{Y}_2^\prime  +  U'  \mathcal{Y}_2 - k^2 U  \gamma +  \gamma^\prime U^\prime\Big)' = ik \Big( ik (U -c) \mathcal{Y}_2 + ik U \gamma^\prime \Big).
\]
This is of course equivalent to \eqref{E:Y2hatformula} if $U_\pm \in C^2$, but makes sense even if $U_\pm^\prime$ has jump discontinuities.  In the latter case, one expects that $\mathcal{Y}_2'$ will likewise exhibit jump discontinuities at the same locations as $U'$.  In particular, Proposition \ref{P:linear} is still valid even if $U_\pm^\prime$ has jump discontinuities in the bulk of the air or water regions, which corresponds to the situation where of $U''$ possesses $\delta$-masses. This justifies our consideration of piecewise linear wind profiles in Section \ref{pathological example section}.
\end{remark}

\begin{remark} \label{linearization remark}
 Let us now revisit the question of how our method differs from the formal linearization procedure.  In \eqref{general LV1}, we demonstrated that $\partial_\alpha v$ will not satisfy the correct boundary conditions on the interface $S$ unless 
 \[ \jump{\partial_{x_2} v^*}z \cdot (-\partial_{x_1} \eta^*, 1)^T - \jump{v^* \cdot \mathbf{e}_1} \partial_{x_1} z = 0.\]
 Note that in the special case of shear flow, this simplifies to 
 \[ \jump{ v^* \cdot \mathbf{e}_1}  = 0,\]
 which is precisely the statement that there is no vortex sheet.  Hence, for a background flow that is continuous over the interface, $\partial_\alpha v$ is indeed the right linearized quantity to consider in the sense that it is in $\mathbb{X}(S)$, but the second one allows for a vortex sheet, or a non-laminar flow, this ceases to be the case.    
 
This is not merely a technical point.  Observe that the tangential velocity does not affect the motion of the fluid interface, and so to truly have a statement about the formation of surface waves, one must guarantee that the instability is for the dynamics of the normal velocity.  Our splitting method is precisely what allows us to do this, and what enables us to see directly what quantities must be controlled in order to ensure stability/instability.  

Lastly, we mention that this has an underlying geometric intuition.  Consider the Lagrangian formulation of the Euler interface problem.  Incompressibility is equivalent to the statement that the restriction of the Lagrangian flow map to each fluid region is volume preserving.  The set of such mappings can be viewed as a submanfiold $\mathcal{M}$ embedded in $L^2(\Omega_0; \rho \, dx)$ (cf., \cite{arnold1966geometrie,brenier1999minimal,ebin1970groups,shatah2008interfaceapriori, shnirelman1985geometry}).  Naturally, the linearized problem about a particular Lagrangian flow map $u_0$ should then be set on the tangent space $T_{u_0} \mathcal{M}$, and the corresponding variation of the Eulerian velocity should lie in $u_0(T_{u_0} \mathcal{M})$.   When one formally linearizes directly in the  Eulerian variables, there is no guarantee that this will be the case because $\mathcal{M}$ is not a flat manifold.  In effect, by thinking exclusively in terms of Eulerian variables, one risks losing a crucial piece of geometric data:  the base point of the tangent space.  Our procedure is carried out in the physical variables, but the splitting is done exactly so that $Y \circ u_0 \in T_{u_0} \mathcal{M}$ and $\nabla \Gamma \circ u_0 \in (T_{u_0} \mathcal{M})^\perp$.  

\end{remark}

\section{Linear instability and critical layers} \label{S:CLayer}

In this section, we consider the physical regime  where 
\[
0< \ep :=  \frac{\rho_+}{\rho_-}  \ll 1, \qquad U_-\equiv 0,
\]
which means the upper fluid (the  air) has much lower density than the lower fluid (the water), and that the lower fluid is stationary.  Moreover, since we are interested in the problem of wind-generation of surface waves, we look for linearized unstable solutions with $\mathpzc{z} \ne0$. Without loss of generality, we normalize by taking 
\[
ik \mathpzc{z}=1. 
\]

\subsection{Derivation of the dispersion relation} Under the above assumptions, $ \gamma$ satisfies 
\be \label{E:gammaeq1}\begin{split} 
& {\gamma}_+ (0) = - \frac{ U_+(0) }{|k| \big( \tanh{(|k|h_+)} + \ep \tanh{(|k|h_-)}\big)}, \quad {\gamma}_+ ^\prime(0) = - |k| \tanh(|k|h_+) 
{\gamma}_+ (0)\\
& {\gamma}_+^\prime (x_2) = |k| {\gamma}_+ (0) \frac {\sinh{(|k|(x_2 - h_+))}} {\cosh{(|k| h_+)}}. 
\end{split} \ee
The boundary conditions for $\mathcal{Y}_2$, which solves equation \eqref{E:Y2hatformula}, take the form 
\be \label{E:Y2hateqBC1} 
\mathcal{Y}_{2\pm} (\pm h_\pm) =0, \qquad  \mathcal{Y}_2(0)
= -\ep |k| \tanh(|k|h_-)  {\gamma}_+ (0)  -c.
\ee
Since $U_- \equiv 0$ in the water, $\mathcal{Y}_{2-}$ can be determined  explicitly: 
\be \label{E:Y2hat-} 
\mathcal{Y}_{2-} (x_2) = \frac { \cosh{(|k|(x_2 + h_-))}} {\cosh{(|k|h_-)}}\mathcal{Y}_2(0), \quad  \mathcal{Y}_{2-}^\prime (0) = |k|\tanh{(|k|h_-)} \mathcal{Y}_{2}(0).
\ee
Therefore, the unknowns for the linearized systems reduce to $(c,  \mathcal{Y}_{2+})$. In addition to boundary conditions \eqref{E:Y2hateqBC1}, $\mathcal{Y}_{2+}$ satisfies 
\[
-  \mathcal{Y}_{2+}^{\prime\prime}  + (\frac {U_+''}{U_+-c} + k^2) \mathcal{Y}_{2+} + \frac {U_+''}{U_+-c} \gamma_+^\prime =0.
\]
We first perform a change of variables to transform this equation into a homogeneous one.  Let 
\[
y= \frac { \mathcal{Y}_{2+} +  \gamma_+^\prime}{\mathcal{Y}_2(0) +  \gamma_+^\prime (0)} = \frac { \mathcal{Y}_{2+} +  \gamma_+^\prime}{U_+(0) -c}.
\]
Then $y$ satisfies 
\be \label{E:yeq1}
- y'' + (\frac {U_+''}{U_+-c} + k^2)y =0 \text{ on } x_2 \in (0, h_+), \qquad y(h_+)=0
\ee
along with the normalizing condition 
\be \label{E:yeq2} 
y(0)=1.
\ee
This is simply the classical Rayleigh's equation that is well-known in the study of the linear instability of shear flows on fixed strips.  Notice, however, that the boundary condition for $y$ is inhomogeneous due to the interface motion. Returning to \eqref{E:ceq}, we see that $-ikc$ is an eigenvalue of the linearized problem if 
\[\begin{split} 
g (1 - \ep) + \frac {\sigma k^2}{\rho_-}  &= -\ep \big(U_+(0) -c\big)  \mathcal{Y}_{2+}^\prime(0) - c  \mathcal{Y}_{2-}^\prime(0) + \ep U_+'(0) \mathcal{Y}_2 (0) \\
&\qquad \qquad \qquad \qquad \qquad - \ep k^2 U_+ (0)  \gamma_+ (0)  + \ep U_+'(0)   \gamma_+^\prime(0). 
\end{split}\]
Substituting \eqref{E:gammaeq1}, \eqref{E:Y2hateqBC1}, and \eqref{E:Y2hat-} into the above equation we obtain
\[\begin{split} 
g (1 - \ep) + \frac {\sigma k^2}{\rho_-}  &= -\ep \big(U_+(0) -c\big)  \mathcal{Y}_{2+}^\prime(0) + c\big(c|k| \tanh{(|k|h_-)} - \ep U_+'(0)\big)  \\
&\qquad +\ep |k| \big(c|k| \tanh^2{(|k|h_-)} -\ep \tanh{(|k|h_-)} U_+'(0) \\
& \qquad  -|k| U_+(0) - U_+'(0) \tanh{(|k|h_+)} \big) \gamma_+ (0). 
\end{split}\]
Finally, in terms of $y$, we have
\be \label{E:yeqBC}\begin{split} 
g (1 - \ep) + \rho_-^{-1} \sigma k^2  &= -\ep \big(U_+(0) -c\big)^2 y'(0) + c^2 |k| \tanh{(|k|h_-)}+\ep U_+'(0) \big(U_+(0) -c\big)\\
& + \frac {\ep c |k| U_+(0)\big(1-\tanh^2{(|k|h_-)}\big)}{\tanh{(|k|h_+)} + \ep \tanh{(|k|h_-)}}.
\end{split}\ee
This is the dispersion relation for the linearized problem with a quiescent ocean and a shear flow in the air.  

We summarize the computations above in the following proposition.
\begin{proposition} \label{P:evalue}
The linearization of the two phase fluid Euler equation at the shear flow $v=U_+(x_2) \chi_{\mathbb{T} \times [0, h_+]} \mathbf{e}_1$ has an eigenvalue $-ikc$ if a solution $y$ of \eqref{E:yeq1} satisfies \eqref{E:yeq2} and \eqref{E:yeqBC}. 
\end{proposition} 

When $\ep=0$, equation \eqref{E:yeqBC} has a pair of solutions for $c$:
\begin{equation} \label{E:ck}
c_k = \sqrt{\frac {g + \rho_-^{-1} \sigma k^2}{|k| \tanh{(|k|h_-)}}} \text{ or } - \sqrt{\frac {g + \rho_-^{-1} \sigma k^2}{|k| \tanh{(|k|h_-)}}}.
\end{equation}
This is simply the dispersion relation of the one phase fluid problem (where the air density is taken as $\rho_+=0$.) For $\ep$ in a neighborhood of $0$, equation \eqref{E:yeqBC} is a quadratic polynomial in $c$ with a complex parameter $\ep y'(0)$ and a real parameter $\ep$, when we fix others like $k$, $h_\pm$, $U_+(0)$, $U_+'(0)$, and so on. The solution $c$ can be expressed analytically in terms of $\ep$ and $\ep y'(0)$ 
\[
c = F\big(\ep y'(0), \ep \big). 
\]
Near $\ep=0$ and $\ep y'(0)=0$, this analytic expression has two branches containing the positive and the negative values of $c_k$ in \eqref{E:ck}, respectively. The quadratic formula clearly implies $F( a, \ep) \in \mathbb{R}$ for small $a\in \mathbb{R}$. Therefore, near $(\ep, \ep y'(0))=(0,0)$, $F$ must take the form 
\be \label{E:ceq1} \begin{split}
c & = F\big(\ep y'(0), \ep\big) \\
& = f_R \big(\ep \realpart{ y'(0) }, \ep \imagpart{y'(0)}, \ep \big) + i \ep \imagpart{y'(0)} f_I  \big(\ep \realpart{y'(0)}, \ep  \imagpart{y'(0)}, \ep\big), \end{split}
\ee
where $f_R$ and $f_I$ are real valued analytic functions satisfying at $(0, 0, 0)$
\be \label{E:ceq2}
f_R  (0) = c_k, \quad \p_1 f_R (0)= f_I(0)= \frac {\big(U_+(0) -c_k\big)^2}{2c_k |k| \tanh{(|k|h_-)}}, \quad \p_2 f_R (0) = 0.
\ee
Here the formula for $f_I(0)$ can be obtained via implicit differentiation.  More detailed information about $f_I$ and $f_R$ can be derived from the quadratic formula if needed. Since we are interested in instabilities, we will seek solutions $(y, c)$ of \eqref{E:yeq1}, \eqref{E:yeq2}, and \eqref{E:yeqBC} with $\imagpart{c} >0$ and $c$ near $c_k$. 
Clearly the key task is to analyze $\imagpart{y'(0)}$, whose dependence on $c$ is quite intricate as it involves solving the Rayleigh equation with a singularity. 

\subsection{Examples}

We first present a few examples where the profile is simple enough to do explicit calculations.  Throughout, $h_\pm=\infty$ 
is assumed for ease of computation. 

\begin{example}[Kelvin--Helmholtz instability] \label{kelvin helmholtz ex} Suppose that $U_+ \equiv U_0$, i.e., the wind velocity is uniform.  Then Rayleigh's equation \eqref{E:yeq1} simplifies to 
\be - y^{\prime\prime} + k^2 y = 0 \textrm{ on } x_2 \in (0, \infty),  \qquad y(0) = 1, \qquad y \to 0 \textrm{ as } x_2 \to \infty.\label{reduced Rayleigh for KH} \ee
This can be solved explicitly.  We find in particular that  
\[ y^\prime(0) = -|k|.\]
Inserting this into \eqref{E:yeqBC} yields the following quadratic equation for $c$:
\[ g(1-\epsilon) + \rho_-^{-1} \sigma k^2 = \epsilon |k| (U_0 - c)^2 + c^2 |k|. \]
Since the coefficients are all real, instability ensues if and only if there are complex roots of this polynomial.  Simply evaluating the discriminant reveals that this will be the case if and only if 
\be 
k^2 U_0^2 \frac{  \rho_+ \rho_-}{(\rho_+ + \rho_-)^2}   > g |k|  \frac{\rho_- - \rho_+}{\rho_+  + \rho_-} + \sigma |k|^3  \frac{1}{\rho_+ + \rho_-}. \label{classic KH} 
\ee
Here we have rearranged terms so that the densities ratios are dimensionless.  This inequality is precisely the Kelvin--Helmholtz instability criterion (cf., e.g., \cite{drazin2004book}). 
\end{example}

\begin{example}[Constant shear without a vortex sheet]  Consider the situation where 
\[ U(x_2) := \left\{ \begin{array}{ll} \mu x_2 & x_2 \geq 0 \\ 0 & x_2 < 0 \end{array} \right., \]
for a fixed $\mu > 0$.  This corresponds to a velocity profile which is continuous over the air--water interface and has a constant (nonzero) shear in the atmosphere.      

As before,  Rayleigh's equation \eqref{E:yeq1} reduces to \eqref{reduced Rayleigh for KH}, and hence $y^\prime(0) = -|k|$.   On the other hand, \eqref{E:yeqBC} becomes   
\[ g (1-\epsilon) + \rho_-^{-1} \sigma k^2 = (1+\epsilon)|k| c^2 - \epsilon \mu c.  \]  
It is completely elementary to show that the above quadratic equation has only real roots when $0 < \epsilon \leq 1$.  We conclude that, for any choice of $\mu$, the corresponding wave is linearly stable.  This is in accordance with Miles's prediction, and our own Theorem \ref{T:Instability}, because $U^{\prime\prime}$ vanishes identically.\end{example}

\begin{example}[Constant shear with a vortex sheet] Building on the previous example, let us now take $U_+$ to be of the form 
\[ U_+(x_2) := U_0 + \mu x_2,\]
for some $U_0 \geq 0$ and $\mu$.  The Rayleigh equation for $y$ is trivial to solve explicitly and we find once more that $y^\prime(0) = -|k|$.  Thus the dispersion relation becomes 
\[ g(1-\epsilon) + \frac{1}{\rho_-}\sigma k^2 = \epsilon (U_0 - c)^2 |k| + \epsilon \mu (U_0 - c) + c^2 |k|.\]
Evaluating the discriminant, we infer that $U$ is unstable if and only if the following inequality is satisfied 
\[ \frac{\epsilon}{(1+\epsilon)^2} U_0 (U_0 + \frac{\mu}{|k|}) > \frac{\epsilon^2}{(1+\epsilon)^2} \frac{\mu^2}{4k^2} + \frac{g}{|k|} \frac{1-\epsilon}{1+\epsilon} + \frac{\sigma |k|}{\rho_-(1+\epsilon)}.\]
Rewriting this in terms of $\rho_\pm$, we get:
\[ k^2 \frac{\rho_+ \rho_-}{(\rho_+ + \rho_-)^2} U_0 ( U_0 + \frac{\mu}{|k|}) > \frac{\rho_+^2}{(\rho_+ + \rho_-)^2} \frac{\mu^2}{4} + g |k| \frac{ \rho_- - \rho_+}{\rho_+ + \rho_-} + \sigma |k|^3 \frac{1}{\rho_+ + \rho_-}.\]
Comparing this to \eqref{classic KH} reveals that vorticity in the air region --- even constant vorticity --- can be destabilizing in the sense that it may reduce the minimal value of $U_0$ required for the  wind-generation of water waves.  

\end{example}

\subsection{The necessity of critical layers for instability.} 
Formal calculations indicates that there exists an unstable eigenvalue $-ikc$ with $c$ near $c_k$ provided that $c_k$ belongs to the range of $U_+$ on $[0, h_+]$, see for example \cite{miles1957windwaves1,janssen2004interaction}. In the following lemma, we prove that this is a necessary condition for the existence of instability near $c_k$. Here, for simplicity, we only consider the case of finite atmosphere $h_+ < \infty$.  

\begin{lemma} \label{L:necessary} 
Suppose $U_+\in C^2$ and $c_k \notin U_+([0, h_+])$, then there exists $\ep_k >0$ such that, if $(y, c)$ solve \eqref{E:yeq1} along with \eqref{E:yeq2} and \eqref{E:yeqBC} for $\ep \in (0, \ep_k)$, and 
\[ |c-c_k| \leq \frac{1}{4} \min_{x_2 \in [0, h_+]} |c_k - U_+(x_2)|,\]
then $c \in \mathbb{R}$.  
\end{lemma}

\begin{proof}
Let $\psi = \psi(x_2)$ be defined by
\[
y(x_2) =: \big(U_+ (x_2) -c\big) \psi(x_2) + \frac {h_+-x_2}{h_+},
\]
and denote
\[
\delta := \min_{x_2 \in [0, h_+]} |c_k - U_+(x_2)|. 
\]
Then \eqref{E:yeq1} implies 
\[
\left\{ \begin{array}{l} -\big((U_+-c)^2 \psi'\big)' + k^2 (U_+-c)^2 \psi + \big(U_+'' + k^2 (U_+ -c)\big)  \dfrac {h_+-x_2}{h_+}=0, \\ \psi(0)=\psi(h_+)=0.\end{array}\right.
\]
Multiplying the above equation by $\bar \psi$ and integrating on $[0,h_+]$, we obtain 
\[
\int_0^{h_+} (U_+-c)^2 (k^{-2} |\psi'|^2 + |\psi|^2) \, dx_2 = \int_0^{h_+} \big( \frac{U_+''}{k^2(U_+-c)} +1\big) \frac {h_+-x_2}{h_+}(U_+-c) \bar \psi \, dx_2.
\]
Even though $c$ may not be real, $|c-c_k|\le \delta/4$ implies $\realpart{(U_+-c)^2} \ge   {\delta^2}/2$. Taking the real part of the above equality, we have 
\[
k^{-1} |\psi'|_{L^2} + |\psi|_{L^2} \le C\delta^{-2},
\]
which implies that
\[ k^{-1} |y'|_{L^2} + |y|_{L^2} \le C\delta^{-2},\]
where $C$ is a constant independent of $\delta$, $\ep$, and $c$. 

Let 
\be \label{E:Wronski1}
\CW := -\frac i2 (\bar y y' - y \bar y')\in \mathbb{R}.
\ee
One may compute 
\be \label{E:Wronski2}
\CW' = \frac {U_+''\imagpart{c}}{|U_+ - c|^2} |y|^2, \qquad \CW(0)=\imagpart{y'(0)},   \quad \CW(h_+)=0.
\ee
Along with the above estimates on $y$, this implies 
\[
|\imagpart{y'(0)}| \le C\delta^{-6} |\imagpart{c}|.  
\]
From \eqref{E:ceq1}, we obtain $|\imagpart{c}| \le C\delta^{-6} \ep |\imagpart{c}|$ and thus $\imagpart{c}=0$. 
\end{proof}

\begin{remark} 
We emphasize that this result, under the $C^2$ assumption on $U_+$, means unstable eigenvalues can only bifurcate out of the imaginary axis from $\pm ikc_k$ in the range of $U$.  If $U_+ \notin C^2$, as our constructions later show, it is entirely possible that there are unstable eigenvalues lurking elsewhere (cf. Section \ref{pathological example section}).

The calculation of the Wronskian \eqref{E:Wronski1} and \eqref{E:Wronski2} will play an important role in the next section where we provide a sufficient condition for instability. 
\end{remark}

\section{Instability induced by critical layers} \label{S:Instability} 

In this section, we present a sufficient condition for linear instability related to critical layers.     We do this by seeking a solution $(y, c)$ of \eqref{E:yeq1} along with \eqref{E:yeq2} and \eqref{E:yeqBC} with $|c-c_k| \ll 1$ and $\imagpart{c}>0$, for $\ep :=  \rho_+/\rho_- \ll 1$. Here $c_k$ is given in \eqref{E:ck} and we are assuming $h_+ < \infty$.  Our main result is the following.  

\begin{theorem} \label{T:Instability}
Assume $U_+ \in C^4$, $h_+<\infty$, and $c_k\in \mathbb{R}$ satisfies  
\[
\emptyset \ne \{ s \in [0, h_+] \mid U_+(s) = c_k\} = \{s_1, \ldots, s_m\} \subset (0, h_+), 
\]
and 
\[
U_+'(s_j) \ne 0, \; c_k U_+''(s_j) \le 0, \; \forall 1\le j \le m, \; \text{ and } \; c_k U_+''(s_j) <0 \text{ for } j=m-1 \text{ or } m. 
\]
For $\ep = \rho_+/\rho_-  \ll 1$, there exists a solution $(y, c)$  of \eqref{E:yeq1} along with \eqref{E:yeq2} and \eqref{E:yeqBC} with $|c-c_k| = O(\ep)$ and $\imagpart{c} >0$ with a positive lower bound   of order $O(\ep)$.
\end{theorem}

\begin{remark} Here $c_k$ may take either its positive or negative value, whichever satisfies the above assumptions. A weaker hypothesis is given later in \eqref{E:CLA3} and the leading order form of $\imagpart{c}$ can be found at the end of Subsection \ref{SS:Instability}.
\end{remark}

Before we give the rigorous argument, it is interesting to outline the  heuristic calculation of $\imagpart{y'(0)}$ which is the key in obtaining the instability due to \eqref{E:ceq1}. The essence of this calculation can be found, for example, in \cite{janssen2004interaction}.  

Assume $c_*\in \mathbb{R}$ satisfies  
\begin{equation} \label{E:CLA1}
\{ s \in [0, h_+] \mid U_+(s) = c_*\} = \{s_1, \ldots, s_m\} \subset (0, h_+), \; U_+'(s_j) \ne 0, \quad j=1, \ldots, m \ge 1
\end{equation}
and $(y, c)$ solves \eqref{E:yeq1} 
with $|c-c_*| \ll 1$, $0<|\imagpart{c}| \ll 1$, and $|y|^2$ reasonably regular. Let 
\[ \{s_1', \ldots, s_m'\} := U_+^{-1} (\{\realpart{c}\}),\]
where $s_j'$ is close to $s_j$. By integrating \eqref{E:Wronski2}, we first obtain 
\begin{align*}
\imagpart{y'(0)} &=\CW (0) \approx - \sgn{(\imagpart{c})}\pi \sum_{j=1}^m \frac {U_+''(s_j') |y(s_j')|^2}{|U_+'(s_j')|} \\
&\approx  - \sgn{(\imagpart{c})} \pi \sum_{j=1}^m \frac {U_+''(s_j) |y(s_j)|^2}{|U_+'(s_j)|},
\end{align*}
where the discrete summation resulted from the limit $\delta$-masses produced by the singularity of the integrand at the critical layers. 

From \eqref{E:ceq1} we find 
\[
\imagpart{c} \approx -\sgn{(\imagpart{c})} \left[ \ep \pi \frac {\big(U_+(0) -c_k\big)^2}{2c_k |k| \tanh{(|k|h_-)}}\sum_{j=1}^m \frac {U_+''(s_j) |y(s_j)|^2}{|U_+'(s_j)|} \right]. 
\]
Therefore, instability occurs if the above square bracketed term is negative. We will make existence of critical layer induced instability rigorous in this section. The crucial part is the analysis near the coefficient singularity of the Rayleigh equation \eqref{E:yeq1}.\\

\noindent {\bf Rayleigh equation} In the process of proving Theorem \ref{T:Instability}, we obtain the following proposition on the convergence of solutions to the Rayleigh equation \eqref{E:yeq1} as the parameter $c$ approaches a limit in $\mathbb{R}$. This result can be useful by itself in the study of instability of shear flows. 

\begin{proposition} \label{P:Rayleigh}
Suppose $c_*\in \mathbb{R}$ is a regular value of $U_+ \in C^l$, $l \ge 4$, on $[0, h_+]$ and $c_* \notin\{ U_+(0), \ U_+(h_+)\}$. We use the notation as in \eqref{E:CLA1}. For $c = c_* + i c_I \in \mathbb{C}\backslash \mathbb{R}$ sufficiently close to $c_*$, let $y$ be the solution of \eqref{E:yeq1} along with $y'(h_+)=1$. We have 
\begin{enumerate} 
\item There exists a unique solution $y_*$ of 
\[
- y_*'' + (\frac {U_+''}{U_+-c} + k^2)y_* =0 \text{ on } x_2 \notin (0, h_+) \cap U_+^{-1} (\{c_*\}), \qquad y_*(h_+)=0, \quad y_*'(h_+) =1
\]
that, at any $s \in U_+^{-1} (\{c_*\})$, exhibits the behavior 
\be \label{E:matching}
y_*(s) = y_*(s\pm) \text{ and } \lim_{x_2 \to 0+} y_*'(s+x_2) - y_*'(s-x_2) = i \sgn(c_I) \frac {\pi U_+''(s)}{|U_+'(s)|} y_*(s). 
\ee
\item For any $s \in U_+^{-1} (\{c_*\})$,
\[
|y_*' (s+ x_2)| = O(|\log |x_2||) \text{ as } x_2 \to 0.
\]
Moreover $y_*$ is $C^{l-3}$ in $c_*$ for $x_2 \notin  U_+^{-1} (\{c_*\})$.
\item For any $\delta >0$ and $\alpha \in (0, 1)$, 
\[
|y(x_2) - y_*(x_2)| = O(|c_I|^\alpha), \quad \textrm{for all }
x_2 \in [0, h_+] \setminus \bigcup_{j=1}^m (s_j-\delta, s_j+ 
\delta).
\]
The above estimates are uniform in $k$ for $k$ on any compact subset of $\mathbb{R}^+$. 
\end{enumerate}
\end{proposition}

\subsection*{Outline of the proof of Theorem \ref{T:Instability}} While the above formal argument provides a useful insight, it is far from straightforward how to turn it into a rigorous proof of Theorem \ref{T:Instability}. Among the issues, for example, are that one needs some control on $y$ and $y'$ for $|\imagpart{c}| \ll 1$, including some positive lower bound on $|y|^2$ to ensure the instability. This can be potentially achieved by identifying the limit of $y$ as $\imagpart{c} \to 0$ along with convergence estimates, but this is very nontrivial due to the creation of $\delta$-masses near critical layers in the limiting process of the singular equation \eqref{E:yeq1}. As the rigorous proof of Theorem \ref{T:Instability} presented in this section is rather technical, here we give a brief outline of the analysis near one singularity $s_0 \in (0, h_+)$ where $|U_+(s_0) - \realpart{c}| + |\imagpart{c}| \ll 1$. 
\begin{itemize} 
\item[{\it Step 1.}] As $\imagpart{c}\to 0$, $y$ and $y'$ do not remain bounded uniformly.  Our first step is thus to understand the behavior of $y(x_2)$ for $x_2$ near $s_0$ when $|\imagpart{c}| \ll 1$. We change variables to $\tau := U(x_2) - \realpart{c}$, which is more convenient for the local analysis due to its appearance in the denominator. Meanwhile the unknowns are transformed to the more geometric quantities $(u_1, u_2, u_3, \CW)$, where $u_1$ and $u_3$ represent the squares of the norms of $y$ and $y'$, and $u_2$ and $\CW$ the dot and cross products of $y$ and $y'$. It turns out that $u_1$ remains uniformly H\"older continuous as $\imagpart{c} \to 0$, while $\CW$ develops a jump discontinuity at $\tau =0$, $u_2$ a logarithmic singularity in $\tau$, and $u_3$ a singularity of the order of the square of logarithm. More careful analysis reveals that in the leading orders the singularities are symmetric in $\tau$ near $\tau=0$. \\
\item[{\it Step 2.}] A priori estimates in {\it Step 1} motivate us to make the right guess for the limit equation \eqref{E:CL5} along with conditions \eqref{E:CL6} at the singular point $\tau=0$. Coefficients and solutions of this limit system still possess singularities. To better understand the limit problem, we apply another linear transformation $B(\tau)$ to the unknowns which depends on $\tau$ {\it smoothly} in neighborhood of $\tau=0$. The resulting system has very simple variable coefficients and can be solved explicitly. \\

\item[{\it Step 3.}] After carefully separating the singular parts of the solutions, and with the help of the above linear transformation $B(\tau)$ applied to solutions $y$ of \eqref{E:yeq1} for $|\imagpart{c}| \ll 1$, we will complete the proof of the Proposition \ref{P:Rayleigh} and obtain very good error estimates near one critical layer.  Finally, the proof of Theorem \ref{T:Instability} is carried out by combining these estimates near all of the critical layers.
\end{itemize}

\subsection{Preliminary estimates near a singularity of Rayleigh's equation} 

Assume $U_+ \in C^3([0, h_+])$ and suppose there is an $s_0 \in (0, h_+)$ and $\delta \in (0,1)$ with 
\begin{equation} \label{E:CLA2} 
 U_+'(s_0)\ne 0, \quad  \frac {|U_+'(x_2)|}{|U_+'(s_0)|} \in (\frac 12, 2) \text{ on } [s_0-\delta, s_0+\delta].
\end{equation}
In this section, we analyze the solutions $y$ of \eqref{E:yeqBC}  on the interval $[s_0 -\delta, s_0 + \delta]$ for $c$ very close to $U_+(s_0)$ with $c_I := \imagpart{c} \ne 0$. For such $c$, there exists $s$ such that 
\[
\realpart{c}=: c_R = U_+(s), \qquad \textrm{ and } \qquad |s-s_0|, \ |c_I |\ll \delta. 
\]
In other words, $s$ is the altitude at which the wind profile takes the value $\realpart{c}$.  

Since \eqref{E:yeq1} is invariant under multiplication by a complex scalar, we make the following transformation which better reflects the rotational invariance on the complex plane: let
\begin{equation} \label{E:CL0.5}
u_1 := |y|^2, \quad u_2 := \frac 12 (y' \bar y + \bar y' y), \quad u_3 := |y'|^2, \quad \CW := - \frac i2 (\bar y y' - y \bar y'). 
\end{equation}
Geometrically, $u_1$ and $u_3$ are the squares of the norms of $y$ and $y'$, while $u_2$ and $\CW$ are the dot and cross products of $y$ and $y'$. Obviously, they satisfy 
\begin{equation} \label{E:CL1}
(u_2)^2 + \CW^2 - u_1 u_3=0 
\end{equation}
and determine the solution $y$ and $y'$ up to a rotation in the complex plane. One of the advantages of these new variables is that identity \eqref{E:Wronski2}, which relates $\CW$ to $|y|^2$, is easier to handle and directly links to the generation of the instability. 

One may compute that $(u_1, u_2, u_3, \CW)$ satisfy
\begin{equation} \label{E:CL2} \begin{cases} 
u_{1}' = 2  u_2\\
u_{2}' =  \left(k^2 + \dfrac {U_+''  \left(U_+ - c_R\right)}{\left(U_+ - c_R\right)^2 + c_I^2} \right)u_1    + u_3 \\
u_{3}' = 2 \left(k^2 + \dfrac {U_+''   \left(U_+ - c_R\right)}{\left(U_+ - c_R\right)^2 + c_I^2} \right) u_2 + \dfrac {2 c_I  U_+''} {\left(U_+ - c_R\right)^2 + c_I^2} \CW\\
\CW' = \dfrac {c_I U_+''}{\big(U_+ - c_R\big)^2 +  c_I^2} u_1
\end{cases}\end{equation}
where $U_+, \ U_+', \ U_+''$ are evaluated at $x_2$. It is straightforward to show that \eqref{E:CL1} is satisfied by solutions of \eqref{E:CL2}. 

To handle the singularity $U_+- c_R=U_+ - U_+(s)$ at $s=0$ for small $c_I$, we work with a new independent variable 
\[ \tau = \tau (x_2, s) := U_+(x_2) -c_R= U_+(x_2) - U_+(s).\]
 This is a valid $C^3$ change of coordinate on $[s_0 - \delta, s_0+\delta]$ depending on  the parameter $s$.  Let $\delta_{1,2} = \delta_{1,2}(s)>0$ be defined by 
\[ -\delta_1(s) := \min\{ U_+(s_0\pm \delta) - U_+(s)\}, \qquad \delta_2(s) := \max\{ U_+(s_0\pm \delta) - U_+(s)\}.\]
Note that this implies that $\delta_{1,2} =O(\delta)$ due to \eqref{E:CLA2} and the fact $|s-s_0| \ll \delta$. For $\tau \in [-\delta_1, \delta_2]$, we have 
\begin{equation} \label{E:CL3} \begin{cases} 
u_{1\tau} = \dfrac 2{U_1}  u_2\\
u_{2\tau} = \dfrac 1{U_1} \left(k^2 + \dfrac {\tau U_2 }{\tau^2 + c_I^2} \right)u_1    + \dfrac {u_3}{U_1} \\
u_{3\tau} = \dfrac 2{U_1} \left(k^2 + \dfrac {\tau U_2}{\tau^2 + c_I^2} \right) u_2 + \dfrac {2c_I  U_2} {(\tau^2 + c_I^2)U_1} \CW\\
\CW_\tau = \dfrac 1{U_1} \dfrac {c_I U_2}{\tau^2 + c_I^2} u_1
\end{cases}\end{equation}
where 
\[
U_1 (s, \tau) := U_+'\big(x_2(\tau, s)\big), \quad U_2 (\tau,s) := U_+''\big(x_2(\tau, s)\big).
\]
Clearly $U_{1,2}$ and $\delta_{1,2} (s)$ are also smooth in $s$. 

We first obtain some uniform estimates on solutions to the transformed system  \eqref{E:CL3}. 

\begin{lemma} \label{L:CL1}
For any $\alpha \in (0, 1)$ and $C_0 \ge 1$, there exist $\delta \in (0,1)$ satisfying \eqref{E:CLA2} and $C, \varepsilon_0>0$ depending only on $\alpha$, $k$, $C_0$, $|U_+'(s_0)|^{-1}$, and $|U_+|_{C^3 ([0, h_+])}$, such that  for all 
\[
|s-s_0|< \varepsilon_0, \; , 0< |c_I| < \varepsilon_0 \text{ and } \tau_{1,2} \in [-\delta_1, \delta_2], \text{ with } |\tau_2| \le C_0 |\tau_1|, 
\]
solutions to \eqref{E:CL3} satisfy  on the interval $[\tau_1, \tau_2]$ (or $[\tau_2, \tau_1]$ if $\tau_2 < \tau_1$)
\[\begin{split}
&|u_1(\tau)|, \; |\CW(\tau)| \le C |u(\tau_1)|_w\\
&|u_2(\tau)| \le  C\big(1+ |\log \frac {c_I^2 + \tau_1^2}{c_I^2 + \tau^2}| \big) |u(\tau_1)|_w, \quad |u_3(\tau) | \le C \big(1 + (\log \frac {c_I^2 + \tau_1^2}{c_I^2 + \tau^2})^2\big)|u(\tau_1)|_w\\
&|u_1(\tau') - u_1(\tau'')| \le  C \Big(  |\log \frac {c_I^2 + \tau_1^2}{c_I^2 + (\tau''')^2}||\tau' - \tau''| + \big(c_I^2 + (\tau''')^2\big)^{\frac {1-\alpha}2} |\tau'-\tau''|^\alpha \Big) |u(\tau_1)|_w, 
\end{split}\]
where $\tau''' \ge \frac 1{C_0}\max\{|\tau'|, |\tau''|\}$ and 
\[
|u(\tau)|_w := \big(u_1(\tau)^2 + u_2(\tau)^2 + u_3(\tau)^2 + \CW(\tau)^2\big)^{\frac 12}.
\]
Moreover, for $\tau \le \min\{|\tau_1|, |\tau_2|\}$,  
\[\begin{split}
&|u_2(\tau) - u_2(-\tau)| \le C \tau_*^{1-\alpha} \tau^\alpha  |u(\tau_1)|_w, \quad   |u_3(\tau) - u_3(-\tau)| \le C |u(\tau_1)|_w
\end{split}\]
where 
\[
\tau_* = (c_I^2 + \tau_1^2)^{\frac 12}.
\]
\end{lemma}

In the above estimates, an almost logarithmic singularity appears at $\tau=0$, but disappears as $\tau$ evolves past $0$ at roughly the same rate at which it appeared. 

\begin{proof} 
For notational simplicity, we mainly consider the case $\tau_1>0$ and $\tau_2  \in [-\tau_1, \tau_1]$, the argument for the other cases being similar (or easier) due to the absolute value outside the logarithm in the above inequalities. See the comments at the end of the proof.

First, we establish an inequality that we will make repeated use of later:
\begin{equation} \label{E:CL4}
\int_{\tau'}^{\tau''} \left|\log \frac {c_I^2 + (\tau''')^2}{c_I^2 + \tau^2}\right|^m d\tau \le C (c_I^2 +(\tau''')^2)^{\frac{1-\alpha}2} |\tau'-\tau''|^\alpha, \text{ where } |\tau'|, |\tau''|\le C_0 |\tau'''|,
\end{equation}
and the constant $C$ depends only on $\alpha \in (0, 1)$, $C_0\ge 1$, and $m>0$ but is independent of $c_I$, $\tau', \tau''$, and $\tau'''$. The above estimate is based on the observation 
\[
\frac {c_I^2 + (\tau''')^2}{c_I^2 + \tau^2} \ge \frac 1{C_0^2} 
\implies \left|\log \frac {c_I^2 + (\tau''')^2}{c_I^2 + \tau^2}\right|^m \le C \left|\frac {c_I^2 + (\tau''')^2}{c_I^2 + \tau^2} \right|^{\frac {1-\alpha}4}.
\]
Applying H\"older's inequality
\[\begin{split}
\int_{\tau'}^{\tau''} \left|\log \frac {c_I^2 + (\tau''')^2}{c_I^2 + \tau^2}\right|^m d\tau &\le C (c_I^2 +(\tau''')^2)^{\frac{1-\alpha}4} |\tau'-\tau''|^\alpha \Big(\int_{\tau'}^{\tau''}  \tau^{-\frac 12} d\tau\Big)^{1-\alpha} \\
&\le C (c_I^2 +(\tau''')^2)^{\frac{1-\alpha}4} |\tau'''|^{\frac{1-\alpha}2} |\tau'-\tau''|^\alpha
\end{split}\]
and thus \eqref{E:CL4}  follows.

For a constant $M\ge 1$, let 
\[
\tau_0 := \sup \{\tau' \in [\tau_2, \tau_1]  \mid |u_2(\tau)| \le M (1+|\log \frac {c_I^2 + \tau_1^2}{c_I^2 + \tau^2}|) |u(\tau_1)|_w, \; \textrm{for all }  \tau \in [\tau', \tau_1]\}.
\]
Clearly $\tau_0< \tau_1$. We will show $\tau_0=\tau_2$ for appropriately chosen $M$ and $\delta$. 
In the rest of the proof, we will use $C$ to denote a generic constant depending only on $C_0$, $\alpha$, $k$, $|U_+'(s_0)|^{-1}$, and $|U_+|_{C^3 ([0, h_+])}$.  For $\tau_0\le \tau' < \tau'' \le \tau_1$, let $\tau'''$ be given with $\tau''' \ge \frac 1{C_0} \max\{|\tau'|, |\tau''|\}$.  Then from \eqref{E:CL3} and \eqref{E:CL4}, we have
\[\begin{split} 
 |u_1(\tau') - u_1(\tau'')| & \le CM |u(\tau_1)|_w \int_{\tau'}^{\tau''} ( 1+ |\log \frac {c_I^2 + \tau_1^2}{c_I^2 + \tau^2}|) \,  d {\tau}\\
&\le  CM |u(\tau_1)|_w \Big( (1+ |\log \frac {c_I^2 + \tau_1^2}{c_I^2 + (\tau''')^2}|)|\tau' - \tau''| + \int_{\tau'}^{\tau''} |\log \frac {c_I^2 + (\tau''')^2}{c_I^2 + \tau^2}| \, d {\tau} \Big)\\
&\le  CM \Big(  |\log \frac {c_I^2 + \tau_1^2}{c_I^2 + (\tau''')^2}||\tau' - \tau''| + \big(c_I^2 + (\tau''')^2\big)^{\frac {1-\alpha}2} |\tau'-\tau''|^\alpha \Big) |u(\tau_1)|_w.
\end{split}\]
This inequality indicates that $u_1$ is H\"older continuous.  It also shows that $|u_1|\le 2 |u(\tau_1)|_w$ on $[\tau_0, \tau_1]$ if $M$ and $\delta$ are chosen to satisfy $C\delta M \le 1$. 

Integrating $\CW_\tau$ on $[\tau, \tau_1]$ with $\tau \in [\tau_0, \tau_1]$ and using the estimate of $u_1$, we have 
\[
|\CW(\tau) - \CW(\tau_1)| =  |\int_{\tau}^{\tau_1} \frac {u_1(\tau') U_2(\tau') }{U_1(\tau')}  \frac {c_I}{(\tau')^2 + c_I^2} \, d\tau'| \le C |u_1|_{C^0} \le C|u(\tau_1)|_w
\]
if $C\delta M \le 1$.  

We will proceed in two steps to estimate $u_{2,3}$. First, for $\tau \in [\max\{0, \tau_0\}, \tau_1]$, \eqref{E:CL3} and the above bounds on $\CW$ imply
\[
|u_3(\tau) - u_3(\tau_1)| \le  C |u(\tau_1)|_w\int_{\tau}^{\tau_1} \left[ M  \big(1+ \frac {\tau'}{(\tau')^2 + c_I^2} \big) \big(1+ |\log \frac {c_I^2 + \tau_1^2}{c_I^2 + (\tau')^2}| \big) + \frac {|c_I|}{(\tau')^2 + c_I^2}\right] \,  d \tau'.
\]
Some terms in the integrand above can be directly integrated, yielding terms like $\tan^{-1}$, $\log$, and $\log^2$.  Ultimately, we find 
\[\begin{split}
|u_3(\tau) - u_3(\tau_1)| \le & 
C \Big(1+ M\big(\tau_*^{1-\alpha} |\tau_1 -\tau|^\alpha + |\log \frac {c_I^2 + \tau_1^2}{c_I^2 + \tau^2}| + (\log \frac {c_I^2 + \tau_1^2}{c_I^2 + \tau^2})^2\big)\Big)|u(\tau_1)|_w.
\end{split}\]
Along with the estimates on $u_{1,3}$,  the $u_2$ equation in \eqref{E:CL3} implies, for $\tau \in [\max\{0, \tau_0\}, \tau_1]$, 
\[\begin{split}
|u_2(\tau) - u_2(\tau_1)| &\le  C \int_{\tau}^{\tau_1} \left[  \big(1+ \frac {\tau'}{c_I^2 + (\tau')^2}\big)  |u(\tau_1)|_w +  |u_3(\tau')| \right] \, d \tau'  \\
&\le C\big(|\log \frac {c_I^2 + \tau_1^2}{c_I^2 + \tau^2}| +  M \tau_*^{1-\alpha} |\tau_1 -\tau|^\alpha \big) |u(\tau_1)|_w \\
&\le C\big( 1+ |\log \frac {c_I^2 + \tau_1^2}{c_I^2 + \tau^2}| \big) |u(\tau_1)|_w,
\end{split}\]
if $CM \delta \le1$. From the above inequalities, we infer that the value $M$ as in the definition of $\tau_0$ can not be achieved on $[\max\{0, \tau_0\}, \tau_1]$ if  $M$ is reasonably large and $\delta$ is chosen such that $C\delta M\le1$. Therefore, either (i) $\tau_0=\tau_2$ in the case $\tau_2 \in [0, \tau_1]$ which completes the proof of the claim $\tau_0=\tau_2$, or else (ii) $\tau_0 < 0$ in the case $\tau_2 <0$. 

Let us consider the latter:  suppose that $\tau_2<0$.  Unfortunately,  one can not simply extend the above argument for $\tau<0$ since the desired logarithmic upper bounds on $u_{2,3}$ start to decrease as $\tau$ decreases past $0$. 
Instead, we study the quantities
\[
\tilde u_{ 2, 3} (\tau) := u_{2,3}(\tau) - u_{2,3}(-\tau),\quad \tau \ge 0,
\]
which satisfy 
\begin{equation} \label{E:CL4.5} \begin{cases} 
\tilde u_{2\tau} = - \dfrac 1{U_1(-\tau)} \tilde u_3 + \big(\dfrac 1{U_1} + \dfrac 1{U_1(-\tau)}\big) u_3 +       \dfrac \tau{\tau^2  + c_I^2} \big(\dfrac{U_2 u_1}{U_1 } -  \dfrac{U_2 u_1}{U_1 } (-\tau)\big) \\
\qquad\qquad + \dfrac {k^2 u_1}{U_1}  + \dfrac {k^2 u_1}{U_1}(-\tau)\\
\tilde u_{3\tau} = \dfrac {2\tau}{\tau^2  + c_I^2}\Big(   \dfrac{U_2}{U_1 } (-\tau) \tilde u_2 + \big(\dfrac{U_2}{U_1}  - \dfrac{U_2}{U_1 } (-\tau)\big) u_2 \Big) + 2k^2 \big(\dfrac{u_2}{U_1} + \frac{u_2}{U_1}(-\tau)\big)  \\
\qquad\qquad+ \dfrac {2 c_I} {(\tau^2 + c_I^2)} \big((\dfrac{U_2\CW}{U_1}) + (\dfrac{U_2\CW}{U_1})(-\tau)\big).
\end{cases}\end{equation}
Here, all functions are evaluated at $\tau$ unless stated otherwise.  Recall that the assumed upper bound on $u_2$, and the estimates derived for $u_1$ and $\CW$, are valid on $[\tau_0, \tau_1]$ and $\tau_0<0$. Using these, with $\tau'''=|\tau_1|$ for $u_1$, we obtain on $[0, \min\{-\tau_0, \tau_1\}]$
\be \label{tilde u2 inequality} \begin{split}
|\tilde u_2| &\le C \int_0^\tau \left( |\tilde u_3(\tau')| + |u_3(\tau')| +  (M \tau_*^{1-\alpha} (\tau')^{\alpha-1}+1) |u(\tau_1)|_w \right) \,  d\tau' \\
& \le  C\left( M \tau_*^{1-\alpha} \tau^\alpha |u(\tau_1)|_w  + \int_0^\tau |\tilde u_3(\tau')| \,  d\tau' \right).
\end{split}\ee
Similarly, the $\tilde u_3$ equation and the assumed upper bound on $u_2$ imply that, for $\tau \in  [0, \min\{-\tau_0, \tau_1\}]$, 
\[\begin{split}
|\tilde u_3| & \le C \int_0^\tau \left( \frac {\tau'}{c_I^2 + (\tau')^2} |\tilde u_2(\tau')| + \big(
M( 1+ |\log \frac {c_I^2 + \tau_1^2}{c_I^2 + (\tau')^2}|) + \frac { |c_I|} {(\tau')^2 + c_I^2} \big) |u(\tau_1)|_w \right) \, d\tau'\\
& \le  C\Big( \int_0^\tau \frac {\tau'}{c_I^2 + (\tau')^2} |\tilde u_2(\tau')| d\tau' + \big( 1 +  M \tau_*^{1-\alpha} \tau^\alpha   \big)  |u(\tau_1)|_w  \Big).
\end{split}\]
Substituting the estimate \eqref{tilde u2 inequality}  for $\tilde u_2$ into the one for $\tilde u_3$ above, taking $CM \delta\le 1$, and then integrating by parts, leads to the following inequality:
\[\begin{split}
|\tilde u_3| \le& C\Big(|u(\tau_1)|_w  + \int_0^\tau M \tau_*^{1-\alpha} (\tau')^{\alpha-1} |u(\tau_1)|_w  + \frac {\tau'}{c_I^2 + (\tau')^2} \int_0^{\tau'} |\tilde u_3(\tau'')| d\tau'' d\tau'\Big)\\
\le & C\Big(|u(\tau_1)|_w  +  \int_0^\tau |\log \frac {c_I^2 + \tau^2}{c_I^2 + (\tau')^2}| |\tilde u_3(\tau')| d\tau'\Big). 
\end{split}\]
Therefore, for $\tau \in [0, \min\{-\tau_0, \tau_1\}]$,
\[\begin{split}
|\tilde u_3|_{C^0([0, \tau])} & \le C \left(|u(\tau_1)|_w  +  \left(\int_0^\tau |\log \frac {c_I^2 + \tau^2}{c_I^2 + (\tau')^2}| \, d\tau' \right) |\tilde u_3|_{C^0([0, \tau])} \right)\\
& \le C \big(|u(\tau_1)|_w  + (c_I^2 + \tau^2)^{\frac {1-\alpha}2} \tau^\alpha |\tilde u_3|_{C^0([0, \tau])} \big),
\end{split}\]
which, along with the fact that $\tilde u_3(0)=0$, implies
\[
|\tilde u_3|_{C^0([0, \tau])} \le C |u(\tau_1)|_w 
\]
for $\tau \in [0, \min\{-\tau_0, \tau_1\}]$. From this we may conclude that
\[
|\tilde u_2(\tau)| \le C M \tau_*^{1-\alpha} \tau^\alpha |u(\tau_1)|_w, \quad \tau \in [0, \min\{-\tau_0, \tau_1\}]. 
\]

Now, choosing $M$ and $\delta$ such that $M\ge C$ and $CM \delta\le 1$, the above inequality and the previous estimate of $u_2$ on $[0, \tau_1]$, show that we must have $\tau_0 \leq \max\{-\tau_1, \tau_2\}$.  Thus, all of the inequalities in statement of the lemma hold on  $[\max\{\tau_2, -\tau_1\}, \tau_1]$, which completes the proof in the case $\tau_2 \in [-\tau_1, \tau_1]$.

We will conclude the proof of the lemma by discussing several other cases of $\tau_{1,2}$. First note that the  case $\tau_1<0$ can be treated in the exactly same manner as the case $\tau_1>0$ because one may consider $u(-\tau)$ and all the estimates go through. The above arguments complete the proof in the case $\tau_1>0$ and $\tau_2 \in [-\tau_1, \tau_1]$. In fact, if $\tau_2\in [0, \tau_1]$,  the estimates on $\tilde u_{2,3}$ are superfluous. The proof for the case $0<\tau_1 \le \tau_2 \le C_0 \tau_1$ is exactly the same (again, there is no need to consider $\tilde u_{2,3}$). Finally, if $\tau_2 \in [-C_0 \tau_1, -\tau_1]$ we can combine the estimates on $[\tau_2, -\tau_1]$ and those on  $[-\tau_1, \tau_1]$ to obtain the desired inequalities. 
\end{proof}

The above lemma provides some preliminary control of solutions to \eqref{E:CL3}. We refine them in the next 
corollary.  
\begin{corollary} \label{C:CL0.5}
For any $\tau_{1,2} \in [-\delta_1, \delta_2]$ with $|\tau_2| \le |\tau_1|$,  the following inequality holds  
\[
|u_1(\tau_1) - u_1 (\tau_2)| \le C (|c_I| + |\tau_1|)^{1-\alpha} \log^2 (|c_I| + |\tau_1|)  |\tau_1 - \tau_2|^\alpha |u(\delta_2)|_w.
\]
Moreover, for $0\le \tau \le \min \{\delta_1, \delta_2\}$, we have 
\[\begin{split}
&|u_2(\tau) - u_2(-\tau)| \le C (|c_I| + |\tau|)^{1-\alpha} \tau^\alpha \log^2 (|c_I| + |\tau|)  |u(\delta_2)|_w\\
&|\CW(\tau) -\CW(-\tau) - \frac {2U_2(0) u_1(0)}{U_1(0) } \tan^{-1} \frac {\tau}{c_I}| + |\CW(\tau) + \CW(-\tau) - 2\CW(0)|\\
&\qquad \qquad \qquad \qquad\qquad\qquad\qquad\qquad  \le  C |c_I| \log^2 |c_I| \log (1+ \frac \tau{|c_I|}) |u(\delta_2)|_w,
\end{split}\]
and
\[\begin{split}
|u_3(\tau)- u_3(-\tau) - \frac{2U_2(0)}{U_1(0)}& \big( \CW(-\tau) + \CW(\tau)   \big) \tan^{-1} \frac {\tau}{c_I} | \le C \tau^\alpha  
|u(\delta_2)|_w.
\end{split}\]
Here $C$ is independent of $c_I$ and $\tau$.  Also, wherever $|u(\delta_2)|_w$ occurs above, it can be replaced by $|u(-\delta_1)|_w$. 
\end{corollary} 

\begin{proof}
Taking $\tau'''= |\tau_1|$ and appealing to Lemma \ref{L:CL1}, we can see that $u_1$ is H\"older continuous on the interval $[\tau_1,\tau_2]$ with exponent $\alpha$ and H\"older constant 
\[
C (c_I^2 + \tau_1^2)^{\frac{1-\alpha}2} |u(\tau_1)|_w 
\le C (|c_I| + |\tau_1|)^{1-\alpha} \log^2 (|c_I| + |\tau_1|) |u(\delta_2)|_w 
\]
where we used the estimates on $|u(\tau)|_w$ in terms of $|u(\delta_2)|_w$. This proves the improved H\"older estimate of $u_1$ in the corollary. The estimate on $u_2(\tau) - u_2(-\tau)$ follows similarly. From \eqref{E:CL3}, we have 
\[
\p_\tau \big(\CW(\tau) - \CW(-\tau) \big) = \frac {c_I}{\tau^2 + c_I^2} \big( (\frac{U_2 u_1}{U_1}) (\tau) + (\frac{U_2 u_1}{U_1}) (-\tau)\big).
\]
Thus, the H\"older continuity of $u_1$ implies 
\[
\p_\tau \big(\CW(\tau) - \CW(-\tau) \big) = \frac {2c_I U_2(0) u_1(0)}{U_1(0) (\tau^2 + c_I^2)} + O\left(\frac {|c_I|}{|c_I| + \tau}\log^2 (|c_I| + \tau)\right) |u(\delta_2)|_w.
\]
For $c_I\ne 0$, $\CW$ is smooth and so we are permitted to integrate this identity yielding 
\[\begin{split} 
& \left|\CW(\tau) - \CW (-\tau)  - 2 \frac {U_2 u_1}{U_1 } (0) \tan^{-1}{(\frac {\tau}{c_I})} \right| \le C |c_I||u(\delta_2)|_w\int_0^\tau 
\frac {\log^2 \big(|c_I| + \tau'\big)}{\big|c_I| + \tau'}  \, d\tau' \\
 \le & C |c_I| |u(\delta_2)|_w \log^2 |c_I| \int_0^{\tau} \frac {1}{|c_I|+\tau'} \, d\tau' \le C |c_I|\log^2 |c_I| \log (1+ \frac \tau{|c_I|})|u(\delta_2)|_w, 
\end{split}\]
and thus we have proved the desired estimate for $\CW(\tau) - \CW(-\tau)$. Again from \eqref{E:CL3}, we have 
\begin{equation} \label{E:CL4.7} \begin{split}
|\p_\tau \big(\CW(\tau) + \CW(-\tau)\big)| &= 
|\frac {c_I}{\tau^2 + c_I^2} \big((\frac {U_2 u_1}{U_1})(\tau) - (\frac {U_2 u_1}{U_1})(-\tau)\big)| \\
&\le  C \frac{|c_I| \log^2 (|c_I| + \tau)}{\tau +|c_I|} |u(\delta_2)|_w
\end{split}\end{equation}
and the desired estimate on $\CW(\tau) + \CW(-\tau)$ follows similarly. 

Finally, to control $u_3(\tau) - u_3(-\tau)$, we use \eqref{E:CL4.5} and \eqref{E:CL4.7}, along with Lemma \ref{L:CL1} (first with $\tau_1=\tau$ and then $\tau_1 =\delta_2$), and the above estimate on $\CW$, to show 
\[\begin{split}
&|\p_\tau\Big( u_3(\tau) - u_3(-\tau) - \frac{2U_2(0)}{U_1(0)} \big( \CW(-\tau) + \CW(\tau)   \big) \tan^{-1}{(\frac {\tau}{c_I})} \Big)| \\
& \qquad \le  C\big( \log^2 (|c_I| + \tau) |u(\delta_2)|_w + \frac {\tau}{c_I^2 + \tau^2} |u_2(\tau) -u_2(-\tau)| + |u(\tau)|_w\big)\\
&\qquad \le  C\big( \log^2 (|c_I| + \tau) |u(\delta_2)|_w + |u(\tau)|_w\big) \le C\log^2 (|c_I| + \tau) |u(\delta_2)|_w. 
\end{split}\]
Integrating this inequality using \eqref{E:CL4} yields the desired estimate for $u_3$. 
\end{proof}

\subsection{Heuristic limit equation near a singularity of  Rayleigh's equation} 

Consider a family of solutions $(u=(u_1, u_2, u_3)^T, \CW )$ parametrized by $c_I\ne 0$.  As $c_I \to 0\pm$, formally, we expect $(u,\CW)$ to converge to a solution $(V = (V_1,V_2,V_3)^T, \mathcal{W}_*)$ of
\begin{equation} \label{E:CL5} \begin{cases} 
V_{\tau} = 
\begin{pmatrix} \dfrac 2{U_1}  V_{2}\\ \dfrac 1{U_1} \Big(k^2 + \dfrac {U_2 }{\tau} \Big)V_{1}    + \dfrac {V_{3}}{U_1} \\ \dfrac 2{U_1} \Big(k^2 + \dfrac {U_2}{\tau} \Big) V_{2} \end{pmatrix} =: \frac 1\tau A(\tau) V, \\
\CW_{*\tau}  =0
\end{cases}\end{equation}
on the interval $\tau \in [-\delta_1, \delta_2]\backslash\{0\}$. At $\tau=0$, Lemma \ref{L:CL1} and Corollary \ref{C:CL0.5} hint that 
\begin{itemize} 
\item $u_1$ has a continuous limit;
\item $\lim_{\tau \to 0}\lim_{c_I \to 0} (u_2(\tau) - u_2(-\tau))=0$;
\item $\CW$ converges to a piecewise constant function with a jump $\frac {\pi U_2 u_1}{U_1} (0)$ at $\tau=0$; and 
\item $\lim_{\tau \to 0}\lim_{c_I \to 0} (u_3(\tau) - u_3(-\tau))$ exists.
\end{itemize}
Therefore, we impose the following conditions at the singularity $\tau=0$:
\begin{equation} \label{E:CL6} \begin{split}
&\lim_{\tau \to 0+} \big(V_{1, 2} (\tau) - V_{1,2} (-\tau) \big) = 0, \quad \CW_*(0+) - \CW_*(0-) = \sgn{(c_I)} \frac {\pi U_2 (0) V_{1} (0)}{U_1(0)},\\
& \lim_{\tau \to 0+} \big( V_{3} (\tau) - V_{3} (-\tau) \big) 
= \sgn{(c_I)} \frac {\pi U_2 (0)}{U_1(0)} \big( \CW_*(0+) + \CW_*(0-)\big).
\end{split}\end{equation} 
Even though $c_I$ is already taken as $0$ here, we still track its sign as the direction from which the real axis is approached does make a significant difference. 

We begin by analyzing the limiting system which will help us justify the convergence and obtain the necessary error estimates. Since $V$ is decoupled from $\CW$ in the limiting system except at $\tau=0$, most of the work can be first carried out for $V$.  

As a preliminary step, in the following lemma, we perform a coordinate change that makes the $ 1/\tau$ singularity in \eqref{E:CL5} more tractable. 

\begin{lemma} \label{L:CL2}
Assume $U_+ \in C^l$, $l\ge 3$, then there exists $B(\tau) \in M_{3\times 3} (\R)$ ($M_{3\times 3} (\R)$ the space of $3\times 3$ real valued matrices) which is $C^{l-2}$ in $\tau \in [-\delta_1, \delta_2]$ such that $B(0)=I$ and for any solution $V$ of \eqref{E:CL5} for $\tau\ne 0$, $\tilde V:= B(\tau)^{-1} V$ satisfies 
\be \label{V tilde ODE} 
\tilde V_{\tau} = \frac{U_2(0)}{\tau U_1(0)} A_0 \tilde V = \tau^{-1} A(0) \tilde V, \quad \tau\ne 0,
\ee
where 
\[
A_0 = \begin{pmatrix} 0&0&0\\1&0&0\\0&2&0 \end{pmatrix}.
\]
Moreover, there exists $\beta \in (0, 1)$ such that $B$ is $C^{l-3, \beta}$ in its dependence on the parameters $k$ and $s$ for $s$ is close to $s_0$. 
\end{lemma}

It is worth pointing out that $B$ is smooth at $\tau=0$ even though the equation is singular there. 

\begin{proof} 
The proof of the lemma is an application of the invariant manifold theory in dynamical systems. Substituting $V = B(\tau) \tilde V$ into \eqref{E:CL5}, we obtain 
\[
\tilde V_{\tau} = \tau^{-1} B(\tau)^{-1} \big( A(\tau) B(\tau) - \tau B' (\tau) \big) \tilde V .
\]
Comparing this to \eqref{V tilde ODE}, we see that it suffices to find $B \in C^{l-2}$ so that $B(0) = I$ and 
\begin{equation} \label{E:CL7}
B' (\tau) = \tau^{-1} \big(A(\tau) B(\tau) - B(\tau) A(0) \big). 
\end{equation}

To remove the singularity $\tau^{-1}$, we treat $\tau$ as a new dimension in the phase space $\R^{10}$ and consider a nonlinear augmented system in a new independent variable $\theta$:
\begin{equation} \label{E:CL8.0} 
\begin{pmatrix} B\\ \tau \end{pmatrix}_\theta = \begin{pmatrix} A(\tau) B - B A(0) \\ \tau \end{pmatrix} =: G(B, \tau). 
\end{equation}
Clearly solutions of \eqref{E:CL8.0} with $\tau \ne 0$ lead to solutions of \eqref{E:CL7}. 

Notice that $(I, 0)$ is a steady state of the augmented system \eqref{E:CL8.0} and 
\[
DG (I, 0) =  \begin{pmatrix} \CA & A_\tau(0) \\ 0 & 1 \end{pmatrix} \; \text{ where the operator $\CA$ is defined as  } \CA B = A(0)B - B A(0).
\]
One may verify through direct computations that 
\[
e^{\theta \CA} B = e^{\theta A(0)} B e^{-\theta A(0)}.
\]
Since 
\begin{equation} \label{E:CL9}
e^{\frac{U_2(0)\theta}{U_1(0)} A_0}=e^{\theta A(0)} = \begin{pmatrix} 1 & 0 & 0 \\ \dfrac{U_2(0)}{U_1(0)} \theta & 1&0 \\ \dfrac{U_2(0)^2}{U_1(0)^2} \theta^2 & \dfrac{2U_2(0)}{U_1(0)} \theta & 1 \end{pmatrix},
\end{equation}
we infer that the solution map $e^{\theta DG(I, 0)}$ at the steady state $(I, 0)$ has the $9$-dimensional invariant subspaces $\{ \tau =0\} \subset \R^{10}$ and $e^{\theta DG(I, 0)}$ has only algebraic growth in $\theta$ there. In the transversal direction, clearly $1$ is an eigenvalue of $DG(I, 0)$. Motivated by the above formula of $e^{\theta \CA}$, one may verify through a routine calculation that $(B_1 , 1)^T$ is a eigenvector of the eigenvalue $1$ where 
\[
B_1 = \int_0^{+\infty} e^{- \theta+ \theta\CA} A_\tau(0) \,  d\theta =  \int_0^{+\infty} e^{- \theta} e^{\theta A(0)} A_\tau(0) e^{-\theta A(0)} \,  d\theta. 
\]
The convergence of this integral is guaranteed by the fact that $e^{\theta \CA}$ has only algebraic growth. 

Therefore, from the standard invariant manifold theory in dynamical systems \cite{CL88, CLL91}, there exists a one-dimensional $C^{l-2}$ locally invariant unstable manifold of $(I,0)$ which is tangent to $(B_1, 1)$ at $(I, 0)$. The regularity of the unstable manifold follows from the observation that  $G \in C^{l-2}$. The smoothness of the unstable manifolds with respect to external parameters --- $s_0$ and $k$ here --- is like the smoothness of the unstable fibers with respect to base points in the center manifold. Therefore it is $C^{l-3, \beta}$ for some $\beta \in (0, 1)$ \cite{CLL91, BLZ08}. The unstable manifold can be parametrized by $\tau$ as $B= B(\tau)$, and will satisfy $B(0)=I$, $B \in C^{l-2}$. The invariance of the one-dimensional unstable manifold implies that $B(\tau)$ is a solution of \eqref{E:CL7}. Since \eqref{E:CL7} is a linear ODE system, the solution $B(\tau)$ can be extend to $\tau \in [-\delta_1, \delta_2]$. In addition, the tangency of the unstable manifold at $(I, 0)$ to $(B_1, 1)$ implies 
\begin{equation} \label{E:CL10}
B' (0)=  \int_0^{+\infty} e^{-\theta} e^{\theta A(0)} A_\tau(0) e^{-\theta A(0)}\,  dr
\end{equation}
which can be integrated explicitly. This complete the proof. 
\end{proof}


As a corollary, we obtain explicit forms of $\tilde V (\tau) = B(\tau)^{-1} V(\tau)$ for any solution $(V, \CW_*)$ of \eqref{E:CL5} and \eqref{E:CL6}. 

\begin{corollary} \label{C:CL1}
Let $(V, \CW_*)$ be a solution of the limiting system \eqref{E:CL5} satisfying the conditions at the singular point in \eqref{E:CL6}, then 
\begin{enumerate} 
\item $V$ is uniquely determined by constants $(a_1, a_2,  a_3)$ along with either $\CW_*(0-)$ or $\CW_*(0+)$ such that $\tilde V(\tau) = B(\tau)^{-1} V(\tau)$ is given by 
\begin{gather*} \tilde V_1 (\tau) = a_1, \qquad \tilde V_2 (\tau) = a_1 \frac{U_2(0)}{U_1(0)} \log |\tau| + a_2, \\
 \tilde V_3 (\tau) = a_1 \frac{U_2(0)^2}{U_1(0)^2} \log^2 |\tau| + 2 a_2 \frac{U_2(0)}{U_1(0)} \log |\tau| + a_3 \pm a_{3}'
\end{gather*}
where the $\pm$ is chosen so that $\pm \tau >0$, and 
\[
a_{3}' 
=\sgn{(c_I)} \frac {\pi U_2 (0)}{2U_1(0)} \big( \CW_*(0+) + \CW_*(0-)\big);
\]
moreover, $V_1(0)= \tilde V_1(0)= a_1$;
\item there exists a constant $C>0$ depending only on $k$, $|U_+|_{C^3([0, h_+])}$, $|U_+(s_0)|^{-1}$, and $\delta$ such that 
\[
\frac 1C \big( |V(-\delta_1)| + |\CW_*(-\delta_1)|  \big) \le |V(\delta_2)| + |\CW_*(\delta_2)| \le C\big( |V(-\delta_1)| + |\CW_*(-\delta_1)|  \big)
\]
 for any solution of \eqref{E:CL5}--\eqref{E:CL6}.
\end{enumerate}
\end{corollary}   

\begin{proof}  From Lemma \ref{L:CL2}, there exist $(a_{1\pm}, a_{2\pm}, a_{3\pm})$, such that 
\[
\tilde V(\tau) = e^{(\log |\tau|) A(0)} (a_{1\pm}, a_{2\pm}, a_{3\pm})^T, \quad \pm \tau >0.
\]
Since $B(0)=I$ and $B(\tau)$ is $C^1$ in $\tau$, we have for $0 < \tau \ll 1$,
\[
V(\tau) -V(-\tau) = B(\tau) \tilde V(\tau) - B(-\tau) \tilde V(-\tau) = \tilde V(\tau) - \tilde V(-\tilde \tau) + O(|\tau| \log^2 |\tau|) 
\]
which implies 
\[
\lim_{\tau \to 0+} \big(V (\tau) - V (-\tau) \big) = \lim_{\tau \to 0+} \big(\tilde V (\tau) - \tilde V (-\tau) \big). 
\]

From the boundary condition \eqref{E:CL6} and the explicit form \eqref{E:CL9} of $e^{r A(0)}$ it is easy to see that $a_{1,2+}=a_{1,2-}$ and the jump condition on $a_{3\pm}= a_3 \pm a_3'$ holds. For $V_1(0)$, since $B(\tau) = I+O(|\tau|)$, we have 
\[
V_1(0) = \lim_{\tau \to 0} V_1(\tau) = \lim_{\tau \to 0} (\tilde V_1(\tau) + O(|\tau| \log^2|\tau|) =a_1. 
\]

Part (2) is a simple consequence of Lemma \ref{L:CL1} and part (1) of this corollary. In fact, 
one can first identify parameters $a_{1,2}$ as well as $\big(a_3 - a_3', \CW_*(0-)\big)$ 
by the values of $(V, \CW_*)$ and $B$ at $-\delta_1$. 
Consequently, condition \eqref{E:CL6} allows us compute the values of $\CW_*(0+) - \CW_*(0-)$ and thus $a_3'$ and $a_3$ using $V_1(0)=\tilde V_1(0) = a_1$. These parameters completely determine the solution on the interval $[-\delta_1, \delta_2]$. Therefore,  the estimates on $V$ and $\CW_*$ follow from the explicit forms of $\tilde V$ and $\CW_*$ and the smoothness of $B$. 
\end{proof} 

Before we finish the analysis on the solutions to the limit system \eqref{E:CL5}--\eqref{E:CL6}, we consider the conservation law \eqref{E:CL1}. Observe that, at this stage, we cannot yet conclude $V_2^2 + \CW_*^2 - V_1 V_3=0$ because $(V, \CW_*)$ is only assumed to be a local solution for $\tau \in [-\delta_1, \delta_2]$. 

\begin{lemma} \label{L:CL2.5}
Let $(V, \CW_*)$ be a solution of \eqref{E:CL5}--\eqref{E:CL6} for $\tau \in [-\delta_1, \delta_2]$. 
\begin{enumerate} 
\item Conservation law \eqref{E:CL1} holds in the sense that
\[
(V_2)^2 + \CW_*^2 - V_1 V_3 =const \; \text{ on } [-\delta_1, \delta_2].
\]
\item Suppose $(V_2)^2 + \CW_*^2 - V_1 V_3=0$ and $V_1(0)=0$, then $(V, \CW_*)$ is smooth on $[-\delta_1, \delta_2]$ and takes the form 
\[
V (\tau) = B(\tau) (0, 0, a_3)^T, \quad \CW_*(0-) = \CW_*(0+)=0.
\]
\end{enumerate}
\end{lemma} 

\begin{proof}
From \eqref{E:CL5}, it is easy to verify through direct computation that $\p_\tau \big( V_2^2 + \CW_*^2 - V_1 V_3\big) =0$ for $\tau \ne 0$, and thus we only need to show 
\[
\big( V_2^2 + \CW_*^2 - V_1 V_3\big)(0-) = \big( V_2^2 + \CW_*^2 - V_1 V_3\big) (0+). 
\]
Again since $B(\tau) = I + O(|\tau|)$ implies $V(\tau) = \tilde V(\tau) + O(|\tau| \log^2 |\tau|)$, it suffices to prove the above identity for $(\tilde V, \CW_*)$, which is straightforward from the explicit form of $\tilde V$ given in part (1) of  Corollary \ref{C:CL1}.

Suppose $V_1(0)=0$, then Corollary \ref{C:CL1} implies $a_1=0$ and thus 
\[
\tilde V(\tau) = (0, a_2, \frac{2a_2 U_2(0)}{U_1(0)} \log|\tau| + a_3 \pm a_3')^T. 
\]
The smoothness of $B$ and the fact that $B(0)=I$ then give  
\[
V(\tau) = (0, a_2, \frac{2a_2 U_2(0)}{U_1(0)} \log|\tau| + a_3 \pm a_3')^T + O(\left|\tau \log |\tau|\right|), \quad \textrm{for } |\tau|\ll 1.
\]
Substituting this into the assumption $V_2^2 + \CW_*^2 = V_1 V_3$, we obtain
\[
a_2^2 + \CW_*^2 + O(\left|\tau \log |\tau|\right|) = O(|\tau| \log^2 |\tau|).
\]
From this we conclude that $a_2=\CW_*(0-)=\CW_*(0+)=0$, and hence part (2) follows. 
\end{proof}

\subsection{Convergence estimates near a singularity of Rayleigh's equation}

Operating under the assumptions in \eqref{E:CLA2}, we now derive an error estimate comparing solutions of \eqref{E:CL3} and the limiting system \eqref{E:CL5}--\eqref{E:CL6} for $0<|c_I| \ll 1$. The following lemma gives some preliminary bounds. 

\begin{lemma} \label{L:CL3} 
Given any solution $(u, \CW )$ of  \eqref{E:CL3} on $[-\delta_1, \delta_2]$ such that $|u(-\delta_1)|_w \le 1$, let $(\underline V, \underline  \CW_*)$ be the unique  solution of \eqref{E:CL5}--\eqref{E:CL6} so that 
\[
\underline \CW_*(0-) + \underline \CW_*(0+) = 2 \CW(0), \quad \underline  V(-\delta_1) = u(-\delta_1).
\]
Then we have 
\[
|u_1 (\tau) - \underline V_1(\tau)| \le C |c_I| \log^2 |c_I|, \quad \tau \in [-\delta_1, \delta_2], 
\]
and 
\[ 
|u(\delta_2) - \underline V(\delta_2)|  \le C |c_I|^\alpha, \quad |\CW (\delta_2) - \underline \CW_* (0+)| + |\CW (-\delta_1) - \underline \CW_* (0-)| \le C |c_I| \left| \log |c_I| \right|^3.
\]
Or, alternatively, we may assume $|u(\delta_2)|_w \le 1$ and the initial condition $\underline{V}(\delta_2) = u(\delta_2)$, which will give the same inequalities except with $\delta_2$ replaced by $-\delta_1$.  \end{lemma} 

Note that in general $u - \underline{V}$ does satisfy the error estimate on the whole interval $[-\delta_1, \delta_2]$ as $u_1 - \underline{V}_1$ does. This is natural since $u_3$ and $\CW$ are smooth for $0< |c_I|$ while $\underline{V}_3$ and $\underline{\CW}_*$ are singular at $\tau=0$. 

\begin{proof} 
Since the linear transformation through $B(\tau)$ simplifies the equation for $V$, which is basically the principle part of the equation satisfied by $(u, \CW)$, we apply the same transformation to the latter by letting $u = B(\tau) \tilde u$. One may compute 
\begin{equation} \label{E:CL11} \begin{cases}
\tilde u_\tau = \dfrac{U_2(0)}{\tau U_1(0)} A_0 \tilde u -\dfrac{c_I^2 U_2}{\tau (\tau^2 + c_I^2) U_1} B^{-1} A_0 B \tilde u + B^{-1} \left(0, 0, \dfrac {2c_I U_2 \CW}{(\tau^2 + c_I^2) U_1}\right)^T \\
\CW_\tau = \dfrac {c_I U_2 u_1}{(\tau^2 + c_I^2) U_1}
\end{cases} \end{equation} 
where $B$ and $U_{1,2}$ are evaluated at $\tau$ unless otherwise specified. 

Ideally, one might try to control the error $\tilde u - \underline {\tilde V} = B(\tau)^{-1} (u - \underline V)$ and $\CW -\underline  \CW_*$. However,  $\tilde u_3$ and $\CW$ are continuous at $\tau=0$ for $c_I\ne 0$, while $\underline {\tilde V}_3$ and $\underline \CW_*$ are not, and this discrepancy makes it inconvenient to estimate $\tilde u - \underline {\tilde V}$ and $\CW - \underline \CW_*$ directly. Instead, we consider the modified error functions $(Z = (Z_1, Z_2, Z_3)^T, \mathbf{W})$ defined for $\tau>0$ by
\begin{equation} \label{E:CL12} \begin{split} 
Z_{1,2}  &: = \tilde u_{1,2}  - \underline {\tilde V}_{1,2}- \left(\tilde u_{1,2} (-\tau) - \underline {\tilde V}_{1,2} (-\tau) \right) \\ &= \tilde u_{1,2}  - \tilde u_{1,2} (-\tau)\\
Z_3  &:= \tilde u_3 - \underline {\tilde V}_3  - \frac {2U_2(0) \underline{\CW}_*(0+)}{U_1(0)} \big(\tan^{-1}(\frac \tau{c_I})  -  \frac \pi2 \sgn{(c_I)}\big) \\
&\qquad  - \left(\tilde u_3 (-\tau) - \underline {\tilde V}_3 (-\tau) -  \frac {2U_2(0) \underline{\CW}_*(0-)}{U_1(0)} \big(- \tan^{-1}{( \frac \tau{c_I})} + \frac \pi2 \sgn{(c_I)}\big) \right)\\
& = \tilde u_3 - \tilde u_3 (-\tau)  - \frac {2U_2(0) }{U_1(0)} \big(\underline \CW_*(0+) +\underline  \CW_*(0-)\big) \tan^{-1}{(\frac \tau{c_I})}  \\
\BW &:= \CW  - \underline \CW_* (0+) + \CW (-\tau) - \underline \CW_* (0-) \\ &= \CW  + \CW (-\tau) - 2 \CW(0),
\end{split}\end{equation}
where all functions are evaluated at $\tau$ unless otherwise stated.  Here we have used the property that $\underline {\tilde V}_{j}(\tau) - \underline {\tilde V}_{j}(-\tau) = 0$ for $j =1, 2$, and, for $j = 3$, is a specified constant given by Corollary \ref{C:CL1}. For any $c_I \ne 0$, $(Z, \BW)$ is smooth for $\tau\ge 0$ and 
\[
|Z(0)|= |\BW(0)|=0. 
\]

From Corollary \ref{C:CL0.5}, we have 
\[
|\BW(\tau)| \le C |c_I| \log^2 |c_I| \log(1+ \frac\tau{|c_I|}).
\]

The next step is to estimate $Z(\tau)$.  Towards that end, we compute that
\[
Z_\tau (\tau) = \tilde u_\tau(\tau) + \tilde u_\tau(-\tau)  - \Big(0, 0, \frac {2 c_I U_2(0)}{U_1(0) (c_I^2 +\tau^2)} \big(\underline \CW_*(0+) + \underline \CW_*(0-) \big) \Big)^T.
\]
Recalling \eqref{E:CL11}, this becomes 
\[\begin{split}
Z_\tau (\tau)&= \frac{U_2(0)}{\tau U_1(0)} A_0 \big( \tilde u (\tau)  - \tilde u (-\tau) \big)  -\frac{c_I^2 }{\tau (\tau^2 + c_I^2)}  \Big(\big( \frac {U_2}{U_1} B^{-1} (0, u_1, 2u_2)^T \big) (\tau) \\
& \quad - \big( \frac {U_2}{U_1} B^{-1} (0, u_1, 2u_2)^T \big) (-\tau) \Big) + \frac {2c_I}{\tau^2 + c_I^2} \Big( \big(B^{-1} (0, 0, \frac {U_2 \CW}{U_1})^T\big)(\tau) \\
&\quad - \big(0, 0, \frac {U_2(0)\underline  \CW_*(0+)}{U_1(0)}\big)^T  + \big(B^{-1} (0, 0, \frac {U_2 \CW}{U_1})^T\big)(-\tau) - \big(0, 0, \frac {U_2(0)\underline  \CW_*(0-)}{U_1(0)} \big)^T \Big).
\end{split}\]
Since $A_0 (0, 0, 1)^T=0$, we see that, for $\tau\ge 0$, $Z$ solves
\[\begin{split}
Z_\tau (\tau) &= \frac{U_2(0)}{\tau U_1(0)} A_0 Z (\tau)+ \frac {2c_I U_2(0)}{U_1(0)(\tau^2 + c_I^2)} \big(0, 0, \BW (\tau)\big)^T +\phi_1(\tau) + \phi_2(\tau),
\end{split}\]
where $\phi_{1,2}$ are given by  
\[
\begin{split} 
\phi_1(\tau) &:=  \frac {2c_I}{\tau^2 + c_I^2} \Big(\big(B^{-1} (0, 0, \frac {U_2 \CW}{U_1})^T\big)(\tau) + \big(B^{-1} (0, 0, \frac {U_2 \CW}{U_1})^T\big)(-\tau) \\
&\qquad - \big(0, 0, \frac {U_2(0)}{U_1(0)} (\CW(\tau) + \CW(-\tau))\big)^T \Big) \\
&= O\big(|c_I| \tau (\tau^2 +c_I^2)^{-1}\big),
\end{split} \]
and 
\[
\begin{split} 
\phi_2(\tau) &:= -\frac{c_I^2 }{\tau (\tau^2 + c_I^2)}  \Big(\big( \frac {U_2}{U_1} B^{-1} (0, u_1, 2u_2)^T \big) (\tau) - \big( \frac {U_2}{U_1} B^{-1} (0, u_1, 2u_2)^T \big) (-\tau) \Big)\\ 
&= c_I^2 O\big( \tau^{\alpha-1} (\tau +|c_I|)^{-(1+\alpha)} \log^2 (|c_I| + \tau)\big) = c_I^2 O\big( \tau^{\alpha-1} (\tau +|c_I|)^{-(1+\alpha)} \log^2 |c_I|\big).
\end{split}\]
Here we used the smoothness of $B$ and Corollary \ref{C:CL0.5}. 

From the equation for $Z_\tau$ and the fact that $e^{rA_0} (0,0,1)^T = (0,0,1)^T$ due to \eqref{E:CL9}, we have 
\[\begin{split}
Z(\tau) &= e^{\frac{U_2(0)}{ U_1(0)} (\log \frac \tau{\tau_0}) A_0} Z(\tau_0) + \int_{\tau_0}^\tau \Big[ \frac {2c_I U_2(0)}{U_1(0)\big((\tau')^2 + c_I^2\big)}\big(0,0, \BW (\tau')\big)^T   \\
&\qquad  +e^{ \frac{U_2(0)}{ U_1(0)} (\log \frac \tau{\tau'}) A_0} \big( \phi_1(\tau') + \phi_2(\tau')\big) \Big] \   d\tau'.
\end{split}\]
Employing the estimates of $\phi_{1,2}$ and $\BW$ derived above, as well as the explicit expression   \eqref{E:CL9} for $e^{rA(0)}$, we then find that  for $\tau>0$
\[\begin{split}
|Z(\tau)| &\le C (1+|\log \frac \tau{\tau_0}|^2) |Z(\tau_0)| +C \int_{\tau_0}^\tau \Big[ \frac{c_I^2 \log^2 |c_I| \log (1+\frac {\tau'}{|c_I|})  }{(|c_I| + \tau')^2} \\
&\qquad + (1+|\log \frac \tau{\tau'}|^2) \big(\frac{ |c_I| \tau'}{(|c_I| + \tau')^2} + \frac{c_I^2 \log^2 |c_I|}{ (\tau')^{1-\alpha}(|c_I| + \tau')^{1+\alpha}}\big)\Big] \,  d\tau'. 
\end{split}\]
When $c_I\ne0$, $Z$ is smooth for $\tau\ge 0$ and thus $Z(\tau_0) = O(|\tau_0|)$ as $\tau_0\to0+$.  This implies
\[\begin{split}
|Z(\tau)| &\le C \int_{0}^\tau \left[\frac{c_I^2 \log^2 |c_I| \log (1+\frac {\tau'}{|c_I|})  }{(|c_I| + \tau')^2} + (1+|\log \frac \tau{\tau'}|^2) \big(\frac{ |c_I| \tau'}{(|c_I| + \tau')^2} + \frac{c_I^2 \log^2 |c_I|}{ (\tau')^{1-\alpha}(|c_I| + \tau')^{1+\alpha}}\big)  \right] \, d\tau'\\
& \le  C |c_I| \int_{0}^{\frac \tau{|c_I|}}  \frac{\tau' + \log^2 |c_I| \log (1+\tau')  }{(1 + \tau')^2} + \frac {\log^2 |c_I|}{(\tau')^{1-\alpha}(1+ \tau')^{1+\alpha}} \\
& \qquad  + \tau^{\frac \alpha2} (|c_I|\tau')^{-\frac \alpha2} 
\big( \frac {\tau'}{(1+\tau')^2} + \frac {\log^2 |c_I| }{(\tau')^{1-\alpha}(1+ \tau')^{1+\alpha}}    \big) \, d\tau'.
\end{split}\]
Integrating the above inequality gives
\[
|Z(\tau)| \le  C |c_I | \big( \log^2 |c_I| +\log (1+\frac \tau{|c_I|}) \big) + C \tau^{\frac \alpha2} |c_I|^{1-\frac \alpha2} \log^2 |c_I|\le C |c_I|^\alpha, 
\]
for $0\le \tau \le \delta_0 := \min\{\delta_1, \delta_2\}$. 
Recalling the definition of $Z$, and using the fact that 
\[ \tan^{-1}{(\frac{\delta_0}{c_I})} - \sgn{(c_I)} \frac \pi2 = O(|c_I|),\]
 we can turn the above inequality into an estimate for $\tilde u$ and $ \underline{\tilde V}$: 
\[
|\tilde u (\delta_0) - \underline {\tilde V} (\delta_0)| \le |\tilde u (-\delta_0) - \underline {\tilde V} (-\delta_0)| + C |c_I|^\alpha.
\]
 Since the system \eqref{E:CL3} is a regular perturbation of \eqref{E:CL5} on $[-\delta_1, \delta_2]\backslash [-\delta_0, \delta_0]$, with the difference in their coefficients being of order $O(|c_I|)$, and noting that we have already established that $\underline {V} (-\delta_1) = u(-\delta_{1})$ and $B(\tau)$ is bounded, it is simple to show that the above inequality implies the desired estimate on $u(\delta_2) - \underline V(\delta_2)$. 

Next, we refine the bound on $u_1$. Observe that the first row of the matrix $A_0$ vanishes. Therefore,  for $\tau\in [-\delta_1, \delta_2]$, equation \eqref{E:CL11}, the property $B(0)=I$, and Lemma \ref{L:CL1} together imply 
\[
|\tilde u_{1\tau}| \le  \frac {C c_I^2}{c_I^2 + \tau^2} |u(\tau)| \le \frac {C c_I^2 \log^2 (|c_I| + |\tau|)}{c_I^2 + \tau^2}.
\]
Arguing as we did before when we integrated $Z_\tau$ in the previous paragraph, we obtain 
\[
|\tilde u_1(\tau) - \tilde u_1(-\delta_1)| \le C |c_I| \log^2 |c_I|. 
\]
From Corollary \ref{C:CL1}, $\underline {\tilde V}_1$ is a constant and thus $\underline {\tilde V}(-\delta_1) =\tilde u(-\delta_1)$.  These facts, together with the boundedness of $B$ and the above inequality lead to the stated inequality for $u_1 - \underline V_1$.

Finally, we complete the proof of the lemma by deriving the estimate on $\CW(\delta_2)$. In fact, the bounds on $\CW$ in Corollary \ref{C:CL0.5}  give
\begin{equation} \label{E:CL13}
|2\CW(\tau) - \frac {2 U_2u_1}{U_1}(0) \tan^{-1}{( \frac \tau{c_I})} - 2 \CW(0) | \le C |c_I| \log^2 |c_I| \log(1+ \frac\tau{|c_I|}). 
\end{equation}
Moreover, the definition of $\underline \CW_*$ and the jump condition on $\underline \CW_*$ at $\tau=0$ given in \eqref{E:CL6} show that 
\[
|\CW(\tau) - \underline \CW_*(0+) - \frac {U_2}{U_1}(0) \big( u_1(0) \tan^{-1}{(\frac \tau{c_I})}  - \sgn{(c_I)} \frac {\pi}{2} \underline V_1(0)\big) | \le C |c_I| \log^2 |c_I| \log(1+ \frac\tau{|c_I|}). 
\]
Thus the estimate on $u_1 - \underline V_1$ implies the desired estimate on $\CW(\delta_2) - \underline \CW_*(0+)$. The analogous inequality for $\CW(-\delta_1) - \underline \CW_*(0-)$ is obtained similarly. 
\end{proof}

Notice that $\underline \CW_*$ in Lemma \ref{L:CL3} does not have the same initial value as $\CW$ at $\tau=-\delta_1$. This can be fixed easily by using Corollary \ref{C:CL1}, as we show in the next corollary. 

\begin{corollary} \label{C:CL2}
Given any solution $\big(u, \CW \big)$ of  \eqref{E:CL3} on $[-\delta_1, \delta_2]$ such that $|u(-\delta_1)|_w \le 1$, let $(V, \CW_*)$ be the unique  solution of \eqref{E:CL5}--\eqref{E:CL6} so that 
\[
(V, \CW_*) (-\delta_1) = (u, \CW) (-\delta_1).
\]
Then we have 
\[
|u_1 (\tau) - V_1(\tau)| \le C |c_I| |\log |c_I||^3, \quad \tau \in [-\delta_1, \delta_2], 
\]
and 
\[ 
|u(\delta_2) -  V(\delta_2)|  \le C |c_I|^\alpha, \quad |\CW (\delta_2) -  \CW_* (0+)| 
\le C |c_I| \left| \log |c_I| \right|^3.
\]
Or, alternatively, assume $|u(\delta_2)|_w \le 1$ and the initial condition on $(V, \CW_*)$ at $\tau= \delta_2$.  Then the same estimate holds only at $\tau = -\delta_1$ rather that $\delta_2$. 
\end{corollary}

\begin{proof} 
Denote by $(\underline V, \underline \CW_*)$ the solution to \eqref{E:CL5}--\eqref{E:CL6} defined in Lemma \ref{L:CL3} and put $\underline {\tilde V} := B^{-1} \underline V$. On the one hand, from \eqref{E:CL6} and Corollary \ref{C:CL1}, we have 
\begin{align*}
\underline \CW_* (0-) &= \frac 12 \big(\underline \CW_* (0+) + \underline \CW_* (0-)\big) - \frac 12\big( \underline \CW_* (0+) -  \underline \CW_* (0-)\big) \\
&= \CW(0) - \sgn{(c_I)} \frac {\pi \underline a_1 U_2(0)}{2U_1(0)},
\end{align*}
where $\underline a_1 := \underline V_1(0) = \underline {\tilde V}_1(0)$. On the other hand, \eqref{E:CL13} and Corollary \ref{C:CL1} imply
\[
|\CW(-\delta_1) - \CW(0) + \sgn{(c_I)} \frac {\pi U_2 u_1}{2 U_1}(0)|\le C |c_I| \left| \log |c_I| \right|^3.
\]
Furthermore, using the estimate on $u_1 - \underline V_1$ given in Lemma \ref{L:CL3}, we see that
\[
|\CW_* (0-) - \CW(0) + \sgn{(c_I)} \frac {\pi \underline a_1 U_2}{2 U_1}(0)|\le C |c_I| \left| \log |c_I| \right|^3. 
\]
Therefore, combining these observations, one arrives at 
\[
|\CW_*(-\delta_1) - \underline \CW_*(-\delta_1)| \le  C |c_I| \left| \log |c_I| \right|^3,
\]
and thus the corollary follows from Lemma \ref{L:CL3} and Corollary \ref{C:CL1}. 
\end{proof} 

Changing from the $\tau$ variable back to the $x_2$ variable, we obtain the other form of the limit system for $(u_*, \CW_*)(x_2) = (V, \CW_*)(\tau(x_2, s))$ on $x_2 \in [s_0-\delta, s_0+\delta]\backslash \{s_0\}$:
\begin{equation} \label{E:CL14} \begin{cases} 
u_{*1}' = 2  u_{*2}\\
u_{*2}' =  \left(k^2 + \dfrac {U_+''}{U_+ - \realpart{c}} \right)u_{*1}    + u_{*3} \\
u_{*3}' = 2 \left(k^2 + \dfrac {U_+''}{U_+ - \realpart{c}} \right) u_{*2}\\
\CW_*' = 0 
\end{cases}\end{equation}
where we recall $\realpart{c} = U_+(s)$. To show that the correct conditions hold at $x_2 =s$, we note that  for $|x_2'| \ll 1$
\[
\tau(s\pm x_2', s) = U_+(s\pm x_2') - U_+(s) = \pm U_+'(s) x_2' + O(|x_2'|^2),
\]
so $u_* := (u_{*1}, u_{*2}, u_{*3})^T$ satisfies 
\[\begin{split}
 u_* (s+ x_2') - u_* (s-x_2') &= V\big(\tau(s+ x_2', s) \big) - V\big(\tau(s-x_2', s)\big)\\
&= \tilde V\big(\tau(s+ x_2', s) \big) - \tilde V\big(\tau(s-x_2', s)\big) + O(|x_2'| \log^2 |x_2'|).
\end{split}\]
Then, since, 
\[
|\tilde V_\tau| = O(|\tau|^{-1}  \left|\log |\tau|\right|), \qquad \tau(s-x_2', s) + \tau(s+ x_2', s) = O(|x_2'|^2) = O(\tau^2),
\] 
we may conclude
\begin{equation} \label{E:CL15}\left\{ \begin{aligned}
&\lim_{x_2' \to 0+} \big(u_{*1, 2} (s+x_2') - u_{*1,2} (s-x_2') \big) = 0, \\ 
&\CW_*(s+) - \CW_*(s-) =  \sgn{(c_I)}\frac {\pi U_+'' (s) u_{*1} (s)}{|U_+'(s)|},\\
& \lim_{x_2' \to 0+} \big( u_{*3} (s+x_2') - u_{*3} (s-x_2') \big) =\sgn{(c_I)}  \frac {\pi U_+'' (s)}{|U_+'(s)|} \big( \CW_*(s+) + \CW_*(s-)\big).
\end{aligned} \right. \end{equation}

At last, we can state the following result on the convergence of solutions of \eqref{E:CL2} to those of the limiting system \eqref{E:CL14}--\eqref{E:CL15}.  This is simply a rephrasing of the previous corollary in light of the change of variable computations above.

\begin{proposition} \label{P:CL1}
Given any solution $\big(u, \CW \big)$ of  \eqref{E:CL2} on $[s-\delta, s+\delta]$ such that $|u(s+\delta)|_w \le 1$, there exists a unique solution $(u_*, \CW_*)$ of (\ref{E:CL14}--\ref{E:CL15}) so that 
\[
(u_*, \CW_*) (s+\delta) = (u, \CW) (s+\delta).
\]
In addition we have 
\[
|u_1 (\tau) - u_{*1}(\tau)| \le C |c_I| |\log |c_I||^3, \quad \tau \in [s-\delta, s+\delta], 
\]
and 
\[ 
|u(s-\delta) -  u_* (s-\delta)|  \le C |c_I|^\alpha, \quad |\CW (s-\delta) -  \CW_* (s-)| 
\le C |c_I| \left| \log |c_I| \right|^3.
\]
Alternatively, we may assume $|u(s+\delta)|_w \le 1$ and impose the initial condition on $(u_*, \CW_*)$ at $x_2=s-\delta$.  Then, one has the estimates above at $x_2 = s+ \delta$ rather than $x_2 = s-\delta$.
\end{proposition}

\subsection{Linear instabilities due to critical layers} \label{SS:Instability}

Suppose $c_*\in \mathbb{R}$ is a regular value of $U_+$ on $[0, h_+]$ satisfying \eqref{E:CLA1}.  As in the statement of Theorem \ref{T:Instability}, let 
\[ U_+^{-1}(\{c_*\}) =: \{s_1 < \ldots <s_m\} \subset (0, h_+), \text{ where } U_+'(s_j) \ne 0.
\]
We consider solutions to \eqref{E:CL2} for $c = c_R + i c_I \in \mathbb{C}\backslash \mathbb{R}$ close to $c_*$. Eventually we restrict our attention to $c_* = c_k$, but the following analysis holds for any regular value $c_*$. 

\begin{proposition} \label{P:CL2}
For any $\alpha \in (0, 1)$ and  $k \in \mathbb{N}$, there exist $C, \varepsilon_0, \delta \in (0, 1]$ depending only on $\alpha$, $k$, $|U_+|_{C^3 ([0, h_+])}$, and $\max_j  |U_+'(s_j)|^{-1}$, such that 
\[
0< s_1 -\delta< s_1 +\delta < s_2 -\delta < \ldots < s_m -\delta < s_m+\delta <h_+ 
\]
and the following estimates hold for  any $c = c_R + ic_I$ satisfying $|c-c_*|, |c_I| \in (0, \varepsilon_0)$. Let  $(u, \CW)$ be the solution of \eqref{E:CL2} with the initial condition 
\[
u(h_+)= (0, 0, 1)^T, \quad \CW(h_+)=0
\] 
and let $(u_*, \CW_*)$ be the solution of \eqref{E:CL14} for $x_2 \notin U_+^{-1} (\{c_R\})$ satisfying 
\[
u_*(h_+) = (0, 0, 1)^T, \quad \CW_*(h_+)=0
\]
and \eqref{E:CL15} at all $s  \in U_+^{-1} (\{c_R\})$. Then we have 
\begin{equation} \label{E:CL17}
|u_1(\tau) - u_{1*} (\tau)| \le C |c_I|^\alpha, \quad \tau \in [0, h_+], 
\end{equation}
and  
\begin{equation} \label{E:CL16}
|u- u_*| + |\CW - \CW_*| \le C |c_I|^\alpha, \quad \textrm{on }   [0, h_+] \backslash \bigcup_{j=1}^m (s_j -\delta, s_j + \delta). 
\end{equation}
\end{proposition}

It is clear that the above solution $(u, \CW)$ corresponds to a solution of \eqref{E:yeq1}. Moreover the existence and uniqueness of $(u_*, \CW_*)$ is due to Corollary \ref{C:CL1}. 

\begin{proof}
When $0< \varepsilon \ll 1$, for any $c= c_R + i c_I$ in an $\varepsilon$-neighborhood of $c_*$, there exist  $s_1'< \ldots < s_m' \in (0, h_+)$ close to $s_1< \ldots < s_m$ (with distance of the order of $O(\varepsilon)$) such that 
\begin{equation} \label{E:CL17.1} 
\{s_1', \ldots, s_m' \}= \{s\in [0, h_+] \mid U_+(s) = c_R\}, \quad U_+'(s_j')  = U_+(s_j) + O(\varepsilon) \ne 0.
\end{equation} 
We first fix $\delta>0$ and take $\varepsilon_0$ sufficiently small so that $|s_j' - s_j|\le \delta/2$ and the estimates in Proposition \ref{P:CL1} hold for $s=s_j'$, $1\le j\le m$. 

For notational convenience, let $s_0 :=-\delta$ and $s_{m+1} :=h_+ +\delta$. By induction on $j$ from $j=m$ to $j=0$, we will prove the desired estimates \eqref{E:CL16} for $u-u_*$ and $\CW-\CW_*$ on the interval $[s_j+ \delta, s_{j+1}-\delta]$ and the estimate  \eqref{E:CL17} for $u_1-u_{1*}$ on $[s_j +\delta, s_{j+1}+\delta] \cap [0, h_+]$.  Together, these imply the proposition. 

Consider $j=m$.  Since $(u, \CW)(h_+) =(u_*, \CW_*) (h_+)$ and \eqref{E:CL3} is a regular perturbation of \eqref{E:CL14} on $[s_m+\delta, h_+]$, as $\varepsilon \to 0+$, inequalities \eqref{E:CL16} and \eqref{E:CL17} hold on $[s_m+\delta, h_+]$ automatically, and actually do so with a better bound $C|c_I|$.  

Suppose we have proved \eqref{E:CL16} on $[s_{j'}+ \delta, s_{j'+1}-\delta]$ and \eqref{E:CL17} on $[s_{j'} +\delta, s_{j'+1}+\delta] \cap [0, h_+]$, for $j'= j+1, \ldots, m$ and $0\le j\le m-1$.  We will now prove them for $j'=j$. Let $(\tilde u, \tilde \CW)$ be the solution of \eqref{E:CL14} with $(\tilde u, \tilde \CW)(s_{j+1} + \delta) = (u, \CW)(s_{j+1} +\delta)$. On the one hand, from Proposition \ref{P:CL1} and the induction hypothesis, we first obtain 
\[\begin{split}
&|(u- \tilde u)(s_{j+1}-\delta)| + |(\CW - \tilde \CW)(s_{j+1} -\delta)| \le C |c_I|^\alpha,\\ 
& |(u_1 - \tilde u_1)(\tau)| \le C |c_I|^\alpha, \; \tau \in [s_{j+1}-\delta, s_{j+1}+ \delta]. 
\end{split}\]
On the other hand, Corollary \ref{C:CL1}  and the induction hypothesis imply
\[
|(u_*- \tilde u)(s_{j+1}-\delta)| + |(\CW_* - \tilde \CW)(s_{j+1} -\delta)| \le C |c_I|^\alpha.  
\]
Moreover, the boundedness of $B(\tau)$ established in Lemma \ref{L:CL2}, along with Corollary \ref{C:CL1}, yields 
\[
|(u_{*1}- \tilde u_1)(\tau)| \le C |c_I|^\alpha, \quad \tau \in [s_{j+1}-\delta, s_{j+1}+ \delta]. 
\]
Therefore \eqref{E:CL16} holds at $x_2 =s_{j+1} -\delta$ and \eqref{E:CL17} holds on $[s_{j+1}-\delta, s_{j+1}+ \delta]$.  Finally, since \eqref{E:CL2} is a regular perturbation of \eqref{E:CL14} on $[s_j+\delta, s_{j+1}-\delta]$ as $\varepsilon \to 0+$, inequalities \eqref{E:CL16} and \eqref{E:CL17} hold on $[s_j+\delta, s_{j+1}-\delta]$ and thus we obtain \eqref{E:CL16} on $[s_{j}+ \delta, s_{j+1}-\delta]$ and \eqref{E:CL17} on $[s_{j} +\delta, s_{j+1}+\delta]$. The proof of the proposition is then completed by induction. 
\end{proof}

In order to obtain a solution of \eqref{E:yeq1} and \eqref{E:yeq2}, we need to prove $u_1(0) >0$ which follows from Proposition \ref{P:CL2} and the next lemma. 

\begin{lemma} \label{L:CL4}
Assume $U_+ \in C^4$.  Let $U_+^{-1} (\{c_R\})=\{s_1', \ldots, s_m'\}$ and say that $(u_*, \CW_*)$ is the solution of \eqref{E:CL14} on $[0,h_+] \setminus \{s_1', \ldots, s_m'\}$ that satisfies 
\[
u_*(h_+) = (0, 0, 1)^T, \quad \CW_*(h_+)=0
\]
and \eqref{E:CL15} at any $s  \in U_+^{-1} (\{c_R\})$. Then $\big(u_*(0), \CW_*(0)\big)$ is $C^1$ in $c_R$ for $c_R$ in a neighborhood of $c_*$. Moreover 
\[
u_{*1, 3} (x_2) \ge 0, \;  x_2 \in [0, h_+]; \quad u_{*1} (s_m')>0, \text{ and if } m\ge 2, \text{ then } u_{*1}(s_{m-1}') >0
\]
where $s_j'$, $j=1, \ldots, m$, are defined in \eqref{E:CL17.1}.
\end{lemma}

\begin{proof}
Since the transformation matrix $B(\tau)$ given in Lemma \ref{L:CL2} is smooth in both $\tau$ and the parameter $s$ close to $s_0$ (in terms of the notations in Lemma \ref{L:CL2}), Corollary \ref{C:CL1} implies the smoothness in $c_R$ of $(u_*, \CW_*)$ near (but not at) a singularity. Equation \eqref{E:CL14} is regular away from $\{s_1', \ldots, s_m'\}$ with uniform bounds on the coefficients. Therefore we see that $\big(u_*(0), \CW(0)\big)$ is $C^1$ in $c_R$ near $c_*$. 

To prove $u_{*1,3} \ge 0$, notice Lemma \ref{L:CL2.5} and the boundary conditions for  $(u_*, \CW_*)$ at $h_+$ imply
\begin{equation} \label{E:CL18} 
u_{*1} u_{*3}- u_{*2}^2 - \CW_*^2 \equiv 0 \quad \text{ on } [0, h_+].
\end{equation}
It follows immediately that $u_{*1}$ and $u_{*3}$ can not vanish simultaneously at any $x_2 \notin U_+^{-1} (\{c_R\})$, and this occurs only when the solution is trivial. Conservation law \eqref{E:CL18} actually further implies that, if one of $u_{*1,3} (x_2) =0$ at some $x_2 \notin U_+^{-1} (\{c_R\})$, it does not change sign since the other one does not vanish in a neighborhood. Moreover, since $u_{*1}$ is continuous on $[0, h_+]$, if $u_{*1} (s_j') \ne 0$, then $u_{*3}$ does not change sign near this $s_j'$ due to \eqref{E:CL18}. Finally, if $u_{*1} (s_j')= 0$, Lemma \ref{L:CL2.5} implies $(u_*, \CW_*)$ is smooth near this $s_j'$ and $u_{*3} (s_j')\ne 0$. Therefore $u_{*1}$ does not change sign near $s_j'$ as well. Summarizing the above discussion, we conclude that $u_{*1,3} \ge 0$ on $[0, h_+]$. 

Finally we prove $u_{*1} (s_m') > 0$ and  $u_{*1}(s_{m-1}') > 0$ if $m\ge 2$. Even though we will continue to work in the framework of \eqref{E:CL14}, the calculation is essentially carried out to the form of the equation used in the proof of Lemma \ref{L:necessary} where $\psi = y/{(U_+-c)}$ was considered. Let 
\[
H(x_2) := u_{*2} - \frac {U_+'}{U_+ - c_R} u_{*1}, \quad x_2 \in [0, h_+] \backslash U_+^{-1} (\{c_R\}).
\]
On the one hand, one may compute 
\[\begin{split}
H'(x_2)&=  k^2 u_{*1} + \big( \frac{U_+'}{U_+-c_R}\big)^2 u_{*1} + u_{*3} - \frac{2U_+'}{U_+-c_R} u_{*2} \\
&\ge  k^2 u_{*1} + \big(\left| \frac{U_+'}{U_+-c_R}\right| \sqrt{u_{*1}} - \sqrt{u_{*3}}\big)^2  \ge 0,
\end{split}\]
in view of \eqref{E:CL18}. Indeed, the conservation law \eqref{E:CL18} also implies that the above derivative vanishes only for the trivial solution. Since $H(h_+)=0$, the monotonicity of $H$ implies 
\be \label{E:CL18.5}
\lim_{x_2 \to s_m'+}  H (x_2) \in [-\infty, 0). 
\ee
On the other hand, from \eqref{E:CL18}, it is clear $H(x_2)=0$ if $u_{*1}(x_2)=0$ at some $x_2 \in [0, h_+] \backslash U_+^{-1} (\{c_R\})$. Moreover, suppose $u_{*1}(s_j')=0$ for some $j=1, \ldots, m$, Lemma \ref{L:CL2.5} and \eqref{E:CL14} imply $u_*(x_2)$, and thus $H(x_2)$ as well, is smooth near $x_2=s_j'$. Moreover \eqref{E:CL18} yields $u_{*2}(s_j')= \CW_*(s_j'\pm)=0$. Therefore  $u_{*1}' (s_j') = 2 u_{*2} (s_j') =0$ which leads to $H(s_j')=0$. Consequently, \eqref{E:CL18.5} implies $u_{*1} (s_m')>0$. From \eqref{E:CL15}, we obtain $\CW_*(s_m'-) \ne \CW_*(s_m+) =0$. Again, Lemma \ref{L:CL2.5} implies $u_{*1} (s_{m-1}') \ne 0$ if $m\ge 2$. This completes the proof of the lemma.  
\end{proof}

\begin{corollary} \label{C:CL3}
Assume in addition to the hypotheses of the previous lemma that $U_+''(s_j)$, $j=1, \ldots, m$, are all non-positive or all non-negative, and $U_+'' (s_{j_0})\ne 0$ for $j_0=m$ or $m-1$.  Then, for any $c_R$ near $c_*$, there exists a unique solution of $(u_*, \CW_*)$ of \eqref{E:CL14} such that 
\[
u_{*1} (0)=1, \; u_{*1, 2} (h_+) =0, \; \CW_*(h_+) =0,
\]
and $(u_*, \CW_*)(0)$ is $C^1$ in its dependence on $c_R$.  Furthermore, $u_{*1}(x_2) \ne 0$ at $x_2 =0, s_j', \ldots, s_m'$ and   
\[
\CW_* (0) = - \sgn{(c_I)} \pi \sum_{j=1}^m \frac {U_+'' (s_j') u_{*1} (s_j')}{|U_+'(s_j')|} \ne 0.
\]
\end{corollary} 

The additional assumption and Lemma \ref{L:CL4} imply $\CW_* \ne 0$ and $\CW_*$ does not change sign on $x_2 < s_{j_0}'$. Consequently, \eqref{E:CL18} and Lemma \ref{L:CL4} imply $u_{*1} (x_2) > 0$ for $x_2 < s_{j_0}'$. The corollary  follows from normalizing the solution furnished by Lemma \ref{L:CL4}.

Finally, combining Proposition \ref{P:CL2} and Lemma \ref{L:CL4} gives the following statement about the solution of Rayleigh's equation \eqref{E:yeq1} with a near singular coefficient.

\begin{proposition} \label{P:CL3}
Under the assumptions in Corollary \ref{L:CL3}, for $c = c_R + i c_I$ sufficiently close to $c_*$ with $c_I \ne 0$, there exists a unique solution $y$ to \eqref{E:yeq1} and \eqref{E:yeq2}, which corresponds to a unique solution $(u, \CW)$  of \eqref{E:CL2} satisfying 
\be \label{E:CL18.6}
u_1(0)=1, \quad u_{1,2}(h_+) = \CW(h_+) =0, \quad u_3(h_+)>0. 
\ee
Let $(u_*, \CW_*)$ be the solution given in Corollary \ref{C:CL3} for $c_R$ and $\sgn(c_I)$. For any $\alpha \in (0, 1)$, there exists $C>0$ such that
\[
|u-u_*| + |\CW - \CW_*| \le C |c_I|^\alpha. 
\]
Moreover 
\[
y'(0) = u_2(0) + i \CW(0). 
\]
\end{proposition}

\begin{proof} 
According to \eqref{E:CL0.5}, any solution $y$ of \eqref{E:yeq1} clearly gives rise to a solution $(u, \CW)$ of \eqref{E:CL2} with the above specified boundary conditions. Such a solution $(u, \CW)$ is unique, much as we saw in Corollary \ref{C:CL3}, in light of Proposition \ref{P:CL2}. 

Conversely, we can reconstruct $y$ from such $(u, \CW)$. In fact, the conservation law \eqref{E:CL1} implies $u_1$ and $u_3$ can not vanish simultaneously, unless the solution is trivial. This and \eqref{E:CL1} further imply that $u_{1,3}$ do not change sign on $[0, h_+]$ and thus both remain nonnegative. 
Let  $\theta = \theta (x_2)$ be defined as the solution of 
\[
\theta' = \frac \CW{u_1}, \quad \theta (0) =0.
\]
This function is well-defined as, if $u_1(x_2)=0$ at some $x_2\in [0, h_+]$, then \eqref{E:CL2} and conservation law \eqref{E:CL1} imply 
\[ u_1'(x_2) = 2 u_2 (x_2)=0 =\CW(x_2) = \CW'(x_2), \qquad  \textrm{and} \qquad  u_1''(x_2 ) = u_3(x_2) > 0.\]

Let $y := \sqrt {u_1} e^{i\theta}.$ One may compute using the definition of $\theta$, \eqref{E:CL1}, and \eqref{E:CL2}, that 
\[
y' = \frac 1{\sqrt{u_1}} (u_2 + i \CW) e^{i\theta},
\]
while 
\[\begin{split}
y'' &= \Big( - u_2 u_1^{-\frac 32} (u_2 + i \CW) + \frac 1{\sqrt{u_1}} \Big(  \big(k^2 + \frac {U_+''  \big(U_+ - c_R\big)}{\big(U_+ - c_R\big)^2 + c_I^2} \big)u_1    + u_3 + \frac {ic_I U_+''}{\big(U_+ - c_R\big)^2 +  c_I^2} u_1\Big) \\
&\qquad +i \CW u_1^{-\frac 32} (u_2 + i \CW) \Big) e^{i\theta}\\
&= \Big( - u_1^{-\frac 32} (u_2^2 + \CW^2) + \frac 1{\sqrt{u_1}} \big(  k^2 u_1 + u_3 + \frac {U_+''}{U_+ - c} \big)u_1 \big)  \Big) e^{i\theta} = \big(k^2 + \frac {U_+''}{U_+ - c} \big) y.
\end{split}\]
Therefore $y$ solves \eqref{E:yeq1}. The estimates on $y$ are from Proposition \ref{P:CL2}.
\end{proof}

\noindent {\bf Proof of Proposition \ref{P:Rayleigh}.} 
The correspondence between the solution $y$ of \eqref{E:yeq1} and $(u, \CW)$ was established in the proof of Proposition \ref{P:CL3},. The properties of $(u, \CW)$ and the convergence estimates were already obtained in the previous lemmas and propositions. To complete the proof, one only needs to confirm that the jump conditions on $(y_*, y_*')$ are satisfied at each $s$ where $U_+(s) = c_*$. In fact, the $\log^2$ growth bound on $|u_{*3}|$ implies $y_*'$ has at most a logarithmic singularity, which shows that $y_*$ is H\"older continuous. This allows us to infer that in the limit $x_2 \to 0+$, 
\begin{align*}
(y_*'\bar y)(s+x_2) - (y_*'\bar y)(s-x_2) 
&=  \big( y_*' (s+ x_2) - y_*'(s-x_2) \big) \bar y_* (s-x_2) \\
& \qquad - y_*'(s+ x_2) \big( \bar y_* (s+ x_2) -\bar y_*(s-x_2) \big) \\
& \to \bar y_* (s) \lim_{x_2 \to 0+} \big( y_*' (s+ x_2) - y_*'(s-x_2) \big).
\end{align*}
As $u_{*2}$ and $\CW_*$ are the real and imaginary parts of $y_*' \bar y_*$, condition \eqref{E:matching} follows from the above calculation and \eqref{E:CL15}. 
\hfill $\square$\\

Finally we are in the position to prove the main theorem. \\

\noindent {\bf Proof of Theorem \ref{T:Instability}.} 
Let $(u_\#, \CW_\#)$ be the solution of the limiting system \eqref{E:CL14} with the parameter $c_R = c_k$ and put
\be
c_\# = - \pi \frac {\big(U_+(0) -c_k\big)^2}{2c_k |k| \tanh{(|k|h_-)}}\sum_{j=1}^m \frac {U_+''(s_j) u_{\#1}(s_j)}{|U_+'(s_j)|}.
\label{def cpound} \ee
Note that $c_\# > 0$ due to Lemma \ref{L:CL4}. Define a mapping $G = (G_1, G_2)(\tilde{c}_1, \tilde{c}_2, \epsilon)$ by 
\begin{align*}
G_1(\tilde c_1, \tilde c_2, \ep) &:=  c_k + \tilde c_1 - f_R (\ep \realpart{y'(0)},\ \ep \imagpart{y'(0)}, \ep) \\
G_2(\tilde c_1, \tilde c_2, \ep) &:=   c_\# +\tilde c_2 - \imagpart{y'(0)} f_I   (\ep \realpart{y'(0)}, \ep \imagpart{y'(0)}, \ep)  
\end{align*}
where $f_{R, I}$ are given in \eqref{E:ceq1} and $y$ is the solution of \eqref{E:yeq1} and \eqref{E:yeq2} with the parameter 
\be \label{c expansion}
c = c_k +\tilde c_1 + i \ep (c_\# + \tilde c_2).
\ee
The existence and uniqueness of $y$ is ensured by Proposition \ref{P:CL3} for small $\tilde c_1$ and $\ep$. Clearly $G$ is smooth for $\ep (c_\#+\tilde c_2)> 0$. In addition, the zero-set of $G$ corresponds to the solutions of \eqref{E:ceq1}, and thus solutions to \eqref{E:yeq1}--\eqref{E:yeqBC}. Proposition \ref{P:CL2} and Corollary \ref{L:CL3} imply
\[ y'(0) = u_2 (0) + i \CW (0) =u_{*2} (0) + i \CW_*(0) + O(\ep^\alpha),\]
where $\alpha$ can be taken arbitrarily in $(0,1)$, $\{s_1', \ldots, s_m'\} := U_+^{-1} (\{c_k + \tilde c_1\})$, and $(u_*, \CW_*)$ is the solution of \eqref{E:CL14} and \eqref{E:CL15} with the parameter $c_k+ \tilde c_1$. From the smoothness of $(u_*(0), \CW_*(0))$ in $\tilde c_1$ due to Corollary \ref{C:CL3}, we have 
\[\begin{split}
y'(0)= &u_{\#2} (0) + i \CW_\#(0) + O(|\tilde c_1|+ \ep^\alpha) 
=  u_{\#2} (0) -  i \pi \sum_{j=1}^m \frac {U_+''(s_j) u_{\#1} (s_j)}{|U_+'(s_j)|} + O(\ep^\alpha + |\tilde c_1|). 
\end{split}\]     

On the other hand, \eqref{E:ceq1} and \eqref{E:ceq2} give us that 
\[
G(\tilde c_1, \tilde c_2, \ep) = \big(\tilde c_1 + O(\ep), \,  \tilde c_2 + O(\ep^\alpha + |\tilde c_1|)\big).  
\] 
Take $(\tilde c_1, \tilde c_2)$ in a rectangle $[-\delta, \delta] \times [-M \delta, M \delta]$, where the fixed constants $\delta$ is small and $M$ is large. By considering the the image of the boundary of this rectangle under the mapping $G (\cdot, \cdot, \ep)$, a standard degree theory argument implies that  $G$ has a zero point near $0$ for any fixed small $\ep>0$. Unravelling definitions, this corresponds to an eigenvalue near $c_k$ with positive imaginary part. The proof of the theorem is  complete. 
\hfill $\square$

\begin{remark} 
1.) To ensure the existence of an instability, the above argument makes clear that one needs only that
\be \label{E:CLA3}
- c_k \sum_{j=1}^m \frac {U_+''(s_j) u_{\#1}(s_j)}{|U_+'(s_j)|} >0,
\ee
and thus those given in Theorem \ref{T:Instability} are sufficient, but not necessary.. \\
2.) If both the positive and negative values of $c_k$ belong to the range of $U_+$ and satisfy the assumption in Theorem \ref{T:Instability}, the above proof implies that there exist at least two distinct unstable modes.
\end{remark}

\section{Instability by other means} \label{pathological example section}

We can summarize the central conclusion of the previous two sections as follows. Assume $U_+\in C^4$ and fix a wave speed $k$. Let $c_k$ denote the corresponding wave speed for the capillary-gravity water wave beneath vacuum problem (i.e. $\ep=0$) given in \eqref{E:ck}. On the one hand, if a sequence of unstable wave speed $c_{k, \ep_n}$ of the water-air problem converges to $c_k$ as the density ratio $\ep_n \to 0+$, then Lemma \ref{L:necessary} implies that there must be a critical layer in the shear flow in the air, i.e., $c_k \in U_+ ([0, h_+])$. On the other hand, under the non-degenerate shear condition on $U_+'$ and some sign condition on $U_+''$, unstable wave speeds may bifurcate from $c_k$ at $\ep=0$. 

The situation for the Kelvin--Helmholtz instability is different. Inequality \eqref{classic KH} indicates that Kelvin--Helmholtz instability occurs if and only if the parameters $(k, \ep)$ are in a region $K \subset \{k,\ \ep >0\}$. With surface tension, there is a positive distance between $K$ and the set $\{\ep=0\}$, that is, for any wind speed $U_0>0$, instability does not occur if $\ep$ is too small. Without surface tension, on the one hand 
for any fixed $k>0$, we have $(k, \ep)\notin K$, i.e., the wave number $k$ is stable if $\ep>0$ is sufficiently small.  In this case, no unstable wave speed bifurcates from $c_k$ at $\ep=0$. On the other hand, the distance between $K$ and $\{\ep=0\}$ is zero, because any sequence $\{(k_n, \epsilon_n)\}  \subset K$ with $k_n \to \infty$ must satisfy $\epsilon_n \to 0$.  One implication of this is that, for any $\epsilon >0$, all sufficiently large modes $k$ are unstable.  In contrast to the water-vacuum setting, this instability ``bifurcates from infinity.''

Relaxing the assumption that $U_+\in C^4$ may lead to additional instability that does not fall into the critical layer theory. 
To illustrate this point, in this section we show that there exist background flows $U$ that are linearly unstable at a wave number $k$ for which $c_k$ is in the range of $U_+$, but the critical layer and support of $U_+^{\prime\prime}$ are separated by a distance uniform in small $\epsilon>0$.  In fact, these are  solutions of the Euler system for which $U_+^{\prime\prime}$ is a $\delta$-measure of negative mass, and the critical layer is at an inflection point. In the view of Remark \ref{R:deltamass}, it is justified to use \eqref{E:yeq1}--\eqref{E:yeqBC} to study the instability of such non-smooth shear flows. 

For simplicity, in these computations we take $\sigma = 0$ and $h_\pm = \infty$.  Consider background profiles of the form
\be U_+(x_2) := \left\{ \begin{array}{ll}  \mu x_2 & 0 \leq x_2 \leq x_2^* \\
 \mu x_2^*  & x_2 > x_2^*. \end{array} \right. \label{inflection:formU} \ee
Here $\mu$ and $x_2^*$ are parameters that we may choose freely.  Notice in particular that 
\[ U_+^{\prime\prime} = -\mu  \delta_{x_2^*}.\] 
Clearly the value 
\[ U_* := U_+(x_2^*) =  \mu x_2^*\] 
 will play an important role. 

\begin{proposition}  Fix a wave speed $k$. Let $\sigma = 0$, and take $h_\pm = \infty$.   There exists $U \in C^{0,1}(\mathbb{R})$ of the form \eqref{inflection:formU} that are unstable in the sense that, for all $0 < \epsilon \ll 1$, there exists $(y, c)$ solving \eqref{E:yeq1}--\eqref{E:yeqBC} with 
\[ 0 < \imagpart{c} = O(\sqrt{\epsilon}).\]
Moreover, for this profile,  $U_* - \sqrt{g/k}$ has a positive lower bound uniform in $\epsilon$, hence the critical layer occurs at an inflection point of $U$ away from $x_2^*$. 
\end{proposition}   

It is worth pointing out that the exponential growth rate $O(\sqrt{\ep})$ of the instability here is much greater than $O(\ep)$ predicted by the critical layer theory. 

\begin{proof}  Taking $U$ as in \eqref{inflection:formU}, it is possible to explicitly solve the Rayleigh equation \eqref{E:yeq1} for $y$.  In general, we find 
\be  y = \left\{ \begin{array}{ll} A_1 e^{-kx_2} & x_2 > x_2^* \\
A_2 e^{-kx_2} + A_3 e^{kx_2} & 0 \leq x_2 \leq x_2^* \end{array} \right. \label{inflection:Yhatgeneralformulat} \ee
where $A_1, A_2, A_3$ are given by the following linear system:
\[ \left\{ \begin{array}{ll}
A_2 + A_3 &= 1 \\
A_1 - A_2 - A_3 e^{2kx_2^*} &= 0 \\
(\dfrac{\mu}{U_* - c} -k) A_1 + k A_2 - k A_3 e^{2k x_2^*} &= 0. \end{array} \right. \]
The first of these is from the boundary condition at $x_2 = 0$, the second is to ensure continuity across $x_2^*$, and the third arises from the jump condition for $y^\prime$ across $x_2^*$.  Solving the system yields 
\begin{gather*}
A_1  = \frac{2k(U_* -c)}{2k(U_*-c) + \mu (e^{-2kx_2^*} - 1)} \\
A_2  = \frac{2k(U_*-c) -\mu }{2k(U_*-c) + \mu (e^{-2kx_2^*} -1)}, \qquad 
A_3  = \frac{\mu e^{-2kx_2^*}}{ 2k(U_*-c) +\mu (e^{-2k x_2^*} - 1)}.
\end{gather*}
Inserting these computations into \eqref{inflection:Yhatgeneralformulat}, we see that
\[ y^\prime(0) = k(A_3 - A_2) = -k \frac{c-\alpha}{c-\beta},\]
where
\be \alpha := U_* - \frac{\mu}{2k}(1+e^{-2kx_2^*}), \qquad \beta = U_* - \frac{\mu}{2k}(1-e^{-2kx_2^*}) .\label{def alpha beta} \ee
With this in hand, the dispersion relation \eqref{E:yeqBC} becomes
\[ g(1-\epsilon) = \epsilon k c^2 \frac{c-\alpha}{c-\beta} + c^2 k - \epsilon c \mu.\]

Note that, by \eqref{def alpha beta}, 
\[ \beta =  U_*\left[ 1 - \frac{1}{2k x_2^*} (1- e^{-2kx_2^*}) \right].\]
The quantity in square brackets on the right-hand side above has range $(0,1)$, with $1$ and $0$ being the limits as $x_2^*$ is taken to $+\infty$ and $0+$, respectively.  It is therefore easy to see that there exist many choices of the parameters $x_2^*$ and $\mu$ for which 
\[ \beta = \sqrt{\frac{g}{k}} < U_* \qquad \textrm{and} \qquad \mu > 0.\]
The dispersion relation can then be simplified into the following polynomial:
\be \label{inflection:cubicc} 
0 = f(c, \epsilon) := \left(c - \sqrt{\frac gk} \right) \left(c^2 - \epsilon \frac {\mu}k c - \frac gk (1-\epsilon) \right)+ \epsilon c^2 (c-\alpha). \ee
Notice that $c = \sqrt{g/k}$ is a double root of $f(\cdot, 0)$.  Indeed, 
\begin{gather*}
f(\sqrt{\frac gk}, 0) =  (\partial_c f) (\sqrt{\frac gk}, 0) = 0, \quad (\partial_c^2 f) (\sqrt{\frac gk}, 0) = 4 \sqrt{\frac gk}>0, \\ 
(\partial_\epsilon f) (\sqrt{\frac gk}, 0) = \frac gk (\sqrt{\frac gk} -\alpha) = \frac{g \mu}{k^2} e^{-2kx_2^*}  >0.
\end{gather*}
Since $f$ is a polynomial of $c$, these facts imply that, for $0 < \epsilon \ll 1$, there exist a complex conjugate pair $c_\pm$ of solutions to $f(c,\epsilon) = 0$ with 
\[
\realpart{c_\pm} = 
\sqrt{\frac gk} + O(\epsilon), \qquad 0< \pm \imagpart{c_\pm} = O(\sqrt{\epsilon}). \qedhere
\] 
\end{proof}

\bibliographystyle{siam}
\bibliography{windwaves}

\end{document}